\documentclass[11pt]{article}

\usepackage{amsmath, amsfonts, amssymb, amsthm,  graphicx, mathtools, enumerate,amsbsy, dsfont}

\usepackage[blocks, affil-it]{authblk}

\usepackage{mathrsfs}

\usepackage[numbers, square, sort]{natbib}
\usepackage[CJKbookmarks=true,
            bookmarksnumbered=true,
			bookmarksopen=true,
			colorlinks=true,
			citecolor=red,
			linkcolor=blue,
			anchorcolor=red,
			urlcolor=blue]{hyperref}
\usepackage[usenames]{color}

\usepackage[letterpaper, left=1.2truein, right=1.2truein, top = 1.2truein, bottom = 1.2truein]{geometry}
\usepackage[ruled, vlined, lined, commentsnumbered]{algorithm2e}

\usepackage{prettyref,soul}

\newtheorem{lemma}{Lemma}[section]

\newtheorem{thm}{Theorem}[section]

\newtheorem{remark}{Remark}[section]

\newrefformat{eq}{(\ref{#1})}
\newrefformat{chap}{Chapter~\ref{#1}}
\newrefformat{sec}{Section~\ref{#1}}
\newrefformat{algo}{Algorithm~\ref{#1}}
\newrefformat{fig}{Fig.~\ref{#1}}
\newrefformat{tab}{Table~\ref{#1}}
\newrefformat{rmk}{Remark~\ref{#1}}
\newrefformat{clm}{Claim~\ref{#1}}
\newrefformat{def}{Definition~\ref{#1}}
\newrefformat{cor}{Corollary~\ref{#1}}
\newrefformat{lmm}{Lemma~\ref{#1}}
\newrefformat{lemma}{Lemma~\ref{#1}}
\newrefformat{prop}{Proposition~\ref{#1}}
\newrefformat{app}{Appendix~\ref{#1}}
\newrefformat{ex}{Example~\ref{#1}}
\newrefformat{exer}{Exercise~\ref{#1}}
\newrefformat{soln}{Solution~\ref{#1}}
\newrefformat{cond}{Condition~\ref{#1}}

\def\text#1{\mbox{\rm #1}}

\def\T{\mathcal{T}}
\def\vec{\text{vec}}

\def\T{{ \mathrm{\scriptscriptstyle T} }}

\def\mn{{ \mathcal{MN} }}
\def\o{{ \mathcal{O} }}
\def\so{{ \mathcal{SO} }}
\def\P{{ \mathcal{P} }}

\newcommand{\argmin}{\mathop{\rm argmin}}
\newcommand{\argmax}{\mathop{\rm argmax}}

\newcommand{\wh}{\widehat}
\newcommand{\wt}{\widetilde}

\newcommand{\fnorm}[1]{\|#1\|_{\rm F}}
\newcommand{\opnorm}[1]{\|#1\|_{\rm op}}

\newcommand{\Tr}{\mathop{\sf Tr}}

\newcommand{\supp}{{\rm supp}}
\newcommand{\iprod}[2]{\left \langle #1, #2 \right\rangle}

\newtheorem*{condb'}{Condition B'}

\newcommand{\br}[1]{\left( #1 \right)}

\newcommand{\mathr}{\mathbb{R}}

\newcommand{\im}{{\rm Im}}
\newcommand{\re}{{\rm Re}}

\newcommand{\diff}{{\rm d}}

\title{Optimal Orthogonal Group Synchronization and Rotation Group Synchronization
}
\author[1]{Chao Gao}
\author[2]{Anderson Y. Zhang}
\affil[1]{
University of Chicago
}
\affil[2]{
University of Pennsylvania
}

\begin{document}
\maketitle

\begin{abstract}
We study the statistical estimation problem of orthogonal group synchronization and rotation group synchronization. The model is $Y_{ij} = Z_i^* Z_j^{*T} + \sigma W_{ij}\in\mathr^{d\times d}$ where $W_{ij}$  is a Gaussian random matrix and $Z_i^*$ is either an orthogonal matrix  or a rotation matrix, and each $Y_{ij}$ is observed independently with probability $p$. We analyze an iterative polar decomposition algorithm for the estimation of $Z^*$ and show it has an error of $(1+o(1))\frac{\sigma^2 d(d-1)}{2np}$ when initialized by spectral methods. A matching minimax lower bound is further established which leads to the optimality of the proposed algorithm as it achieves the exact minimax risk.

\smallskip

\end{abstract}


\section{Introduction}\label{sec:intro}

Consider
\begin{equation}
Y_{ij}=Z_i^*Z_j^{*\T}+\sigma W_{ij}\in\mathbb{R}^{d\times d}, \label{eq:od-model}
\end{equation}
for all $1\leq i<j\leq n$. We assume
\begin{equation}
Z_i^*\in\o(d)=\{U\in\mathbb{R}^{d\times d}:UU^{\T}=U^{\T}U=I_d\}, \label{eq:od-setting}
\end{equation}
for all $i\in[n]$ and $W_{ij}\sim \mathcal{MN}(0,I_d,I_d)$ independently for all $1\leq i<j\leq n$.\footnote{A random matrix $X$ follows a matrix Gaussian distribution $\mathcal{MN}(M,\Sigma,\Omega)$ if its density function is proportional to $\exp\left(-\frac{1}{2}\Tr\left(\Omega^{-1}(X-M)^{\T}\Sigma^{-1}(X-M)\right)\right)$.} Our goal is to estimate the orthogonal matrices $Z_1^*,\cdots,Z_n^*$. This problem is known as orthogonal group synchronization, or $\o(d)$ synchronization. In addition to (\ref{eq:od-setting}), we also consider a closely related setting that
\begin{equation}
Z_i^*\in\so(d)=\{U\in\o(d):\det(U)=1\},\label{eq:sod-setting}
\end{equation}
for all $i\in[n]$. This is known as rotation group synchronization, or $\so(d)$ synchronization. Both $\o(d)$ and $\so(d)$ synchronizations have found successful applications across a wide range of areas including structural biology, computational genomics, robotics, computer vision and distributed networks. For example, synchronization over $\o(d)$ has been applied to the sensor network localization problem \citep{cucuringu2012sensor}. The problem over $\o(3)$ can be used to solve the graph realization problem \citep{cucuringu2012eigenvector}, and that over $\so(3)$ plays a central role in cryo-electron microscopy \citep{singer2011three,shkolnisky2012viewing} and global motion estimation \citep{arie2012global}.

Despite a growing literature in application and methodology, theoretical understandings of synchronization over $\o(d)$ or $\so(d)$ have not been thoroughly explored. In particular, the exact minimax estimation of the $\o(d)$ synchronization under the model (\ref{eq:od-model}) still remains an important open problem. In this paper, we carefully characterize the minimax risk with respect to the following loss function,
\begin{equation}
\ell(Z,Z^*)=\min_{B\in\o(d)}\frac{1}{n}\sum_{i=1}^n\fnorm{Z_i-Z_i^*B}^2, \label{eq:loss-od-def}
\end{equation}
defined for all $Z,Z^*\in\o(d)^n$. Note that the minimization over $B\in\o(d)$ is necessary, since multiplying every $Z_i^*$ by some $B\in\o(d)$ does not change the distribution of the observations. Our result is obtained under a setting that allows the possibility of missing interactions. Instead of observing $Y_{ij}$ for all $1\leq i<j\leq n$, we assume that each $Y_{ij}$ is observed with probability $p$. In other words, we only observe (\ref{eq:od-model}) on a random graph that $A_{ij}\sim\text{Bernoulli}(p)$ independently for all $1\leq i<j\leq n$. We summarize the main result of the paper on the $\o(d)$ synchronization as the following theorem.

\begin{thm}\label{thm:main}
For the $\o(d)$ Synchronization (\ref{eq:od-setting}), assume $\frac{np}{\sigma^2}\rightarrow\infty$, $\frac{np}{\log n}\rightarrow\infty$ and $2\leq d=O(1)$. Then, there exists some $\delta=o(1)$ such that
\begin{equation}
\inf_{\wh{Z}\in\o(d)^n}\sup_{Z\in\o(d)^n}\mathbb{E}_Z\ell(\wh{Z},Z)\geq (1-\delta)\frac{\sigma^2d(d-1)}{2np}.
\end{equation}
Moreover, the algorithm $\wh{Z}$ described in Section \ref{sec:od-upper} satisfies
\begin{equation}
\ell(\wh{Z},Z^*)\leq (1+\delta)\frac{\sigma^2d(d-1)}{2np}, \label{eq:intro-upper}
\end{equation}
with probability at least $1-n^{-8}-\exp\left(-\left(\frac{np}{\sigma^2}\right)^{1/4}\right)$.
\end{thm}
Though the result of Theorem \ref{thm:main} is stated in asymptotic forms, non-asymptotic versions under the assumptions $\frac{np}{\sigma^2}\geq c_1$ and $\frac{np}{\log n}\geq c_2$ for some sufficiently large constants $c_1,c_2>0$ are also presented in the paper, with the exact form of $\delta$ explicitly given in Theorem \ref{thm:upper-od} and Theorem \ref{thm:lower-od}.
The high-probability upper bound (\ref{eq:intro-upper}) immediately implies an in-expectation upper bound given the boundedness of the loss function that $\ell(Z,Z^*)\leq 4d$ for all $Z,Z^*\in\o(d)^n$. Since $\exp\left(-\left(\frac{np}{\sigma^2}\right)^{1/4}\right)=o\left(\frac{\sigma^2d(d-1)}{2np}\right)$, we have
\begin{equation}
\sup_{Z\in\o(d)^n}\mathbb{E}_Z\ell(\wh{Z},Z)\leq (1+\delta)\frac{\sigma^2d(d-1)}{2np} + \frac{4d}{n^8}, \label{eq:exp-upper}
\end{equation}
for some $\delta=o(1)$. According to the proof of Theorem \ref{thm:main}, the $4dn^{-8}$ in (\ref{eq:exp-upper}) can actually be improved to $4dn^{-C}$ for any constant $C>0$. Therefore, if we additionally assume that $\sigma^2/p\geq n^{-c}$ for some constant $c>0$, we will have
$$(1-\delta)\frac{\sigma^2d(d-1)}{2np}\leq \inf_{\wh{Z}\in\o(d)^n}\sup_{Z\in\o(d)^n}\mathbb{E}_Z\ell(\wh{Z},Z)\leq (1+\delta)\frac{\sigma^2d(d-1)}{2np},$$
for some $\delta=o(1)$. Hence, $\frac{\sigma^2d(d-1)}{2np}$ is the exact asymptotic minimax risk for $\o(d)$ synchronization. We remark that $\frac{\sigma^2d(d-1)}{2np}$ is intuitive to understand, since $\sigma^2$ is the noise level, $np$ is the effective sample size, and $\frac{d(d-1)}{2}$ is the degrees of freedom of $\o(d)$. In additional to Theorem \ref{thm:main}, we also obtain a very similar result for the minimax risk of $\so(d)$ synchronization. See Theorem \ref{thm:upper-sod} and Theorem \ref{thm:lower-sod} for the exact statement.

To achieve the minimax optimality, we consider an iterative polar decomposition procedure that projects the matrix $\sum_{j\in[n]\backslash\{i\}}A_{ij}Y_{ij}Z_j^{(t-1)}$ to the space $\o(d)$ at the $t$th iteration. This algorithm can be viewed as an approximation to the maximum likelihood estimator, and is known under other names such as generalized power method \citep{boumal2016nonconvex,gao2019iterative,liu2020unified,ling2020improved} and projected power method \citep{chen2018projected} in the literature. We establish a sharp statistical error bound for the evolution of the algorithm, and shows that the error decays exponentially to the optimal $\frac{\sigma^2d(d-1)}{2np}$ as long as the algorithm is initialized by a spectral method \citep{arie2012global,singer2011three}. Our lower bound analysis is a careful application of the celebrated van Trees' inequality \citep{gill1995applications}. It complements the Cram{\'e}r--Rao lower bound derived by \cite{boumal2014cramer} for the set of unbiased estimators.

Let us give some very brief comments on the assumptions of Theorem \ref{thm:main}. Our paper is focused on the setting where $d$ does not grow with the sample size $n$. This covers the most interesting applications in the literature for $d=3$, though an extension of our result to a growing $d$ would also be theoretically interesting. We exclude the case $d=1$, because $\so(1)$ is a degenerate set, and the problem over $\o(1)=\{-1,1\}$ is known as $\mathbb{Z}_2$ synchronization, whose minimax rate has already been derived in the literature \citep{fei2020achieving,gao2021sdp}. It is interesting to note that the minimax rate of $\mathbb{Z}_2$ synchronization is exponential instead of the polynomial rate of $\o(d)$ synchronization for $d\geq 2$. When $d=O(1)$, the condition $\frac{np}{\sigma^2}\rightarrow\infty$ is equivalent to the minimax risk being vanishing. Since the loss function is bounded, the condition $\frac{np}{\sigma^2}\rightarrow\infty$ can also be viewed as necessary for the minimax risk to have a nontrivial rate. We remark that a nontrivial estimation when $\frac{np}{\sigma^2}\asymp 1$ is still possible, but that requires a very different technique of approximate message passing (AMP) \citep{perry2018message} that does not apply to $\frac{np}{\sigma^2}\rightarrow\infty$. Finally, the condition $\frac{np}{\log n}\rightarrow\infty$ guarantees that the random graph is connected with high probability so that synchronization up to a global phase ambiguity is possible.

\paragraph{Related Literature.} One popular method for group synchronization is semi-definite programming (SDP) \citep{arie2012global,singer2011three}. The tightness of SDP and other forms of convex relaxation has been studied by \cite{ling2020solving,fan2021joint,wang2013exact}. In particular, it is shown by \cite{ling2020solving} that SDP is tight for $\o(d)$ synchronization when $\sigma^2\lesssim\sqrt{n}$ in the setting of $p=1$. The papers \cite{arie2012global,singer2011three,romanov2020noise,boumal2013robust} have studied spectral methods and its asymptotic error behavior. In terms of statistical estimation error, \cite{ling2020improved} and \cite{ling2020near} have derived error bounds for the generalized power method and the spectral method for $\ell_{\infty}$-type loss functions in the setting of $p=1$ with a general $d$ that can potentially grow. In particular, both rates are $\frac{\sigma^2\log n}{n}$ when $d$ is bounded by some constant.  For partial observations, the analysis of \cite{liu2020unified} for the generalized power method applied to $p<1$, but they require $p\gtrsim n^{-1/2}$ for a nontrivial result.

When $d=2$, $\so(d)$ synchronization is also known as angular/phase synchronization, and has been extensively studied in the literature \citep{singer2011angular}. The tightness of SDP has been established by \cite{bandeira2017tightness,zhong2018near}. The convergence property and statistical estimation error of the generalized power method are studied by \cite{zhong2018near,liu2017estimation}. The work that is mostly related to us is \cite{gao2020exact} that derives the minimax risk of phase synchronization. The analysis of \cite{gao2020exact} relies critically on the representation of an $\so(2)$ element as a unit complex number. However, as soon as $d\geq 3$, elements of the groups $\o(d)$ and $\so(d)$ are general non-commutative matrices, and the techniques in \cite{gao2020exact} cannot be applied to derive Theorem \ref{thm:main}. See Section \ref{sec:oracle} for a detailed discussion.

\paragraph{Paper Organization.} The rest of the paper is organized as follows. In Section \ref{sec:upper}, we analyze the error decay of the iterative polar decomposition algorithm and the statistical property of the initialization procedure. This leads to the upper bound results for $\o(d)$ and $\so(d)$ synchronizations. The lower bounds are derived in Section \ref{sec:lower}. Finally, Section \ref{sec:pf} collects all the technical proofs of the paper.

\paragraph{Notation.} For $d \in \mathbb{N}$, we write $[d] = \{1,\dotsc,d\}$.  Given $a,b\in\mathbb{R}$, we write $a\vee b=\max(a,b)$ and $a\wedge b=\min(a,b)$. For a set $S$, we use $\mathbb{I}\{S\}$ and $|S|$ to denote its indicator function and cardinality respectively. The notation $\mathds{1}_d$ denotes a vector of all ones.
For a matrix $B =(B_{ij})\in\mathbb{R}^{d_1\times d_2}$, the Frobenius norm, matrix $\ell_{\infty}$ norm, and operator norm of $B$ are defined by $\fnorm{B}=\sqrt{\sum_{i=1}^{d_1}\sum_{j=1}^{d_2}|B_{ij}|^2}$, $\|B\|_{\ell_{\infty}}=\max_{1\leq i\leq d_1}\sum_{j=1}^{d_2}|B_{ij}|$ and $\opnorm{B} = s_{\max}(B)$, and we use $s_{\min}(B)$ and $s_{\max}(B)$ for the smallest and the largest singular values of $B$. For $U,V\in\mathbb{R}^{d_1\times d_2}$, $U\circ V\in\mathbb{R}^{d_1\times d_2}$ is the Hadamard product $U\circ V=(U_{ij}V_{ij})$, and the trace inner product is $\Tr(UV^{\T})=\sum_{i=1}^{d_1}\sum_{j=1}^{d_2}U_{ij}V_{ij}$. The notation $\det(\cdot)$, $\vec(\cdot)$ and $\otimes$ are used for determinant, vectorization, and Kronecker product. For two integers $d_1\geq d_2$, define $\o(d_1,d_2)=\{U\in\mathbb{R}^{d_1\times d_2}:U^{\T}U=I_{d_2}\}$ so that $\o(d)=\o(d,d)$.
The notation $\mathbb{P}$ and $\mathbb{E}$ are generic probability and expectation operators whose distribution is determined from the context.  For two positive sequences $\{a_n\}$ and $\{b_n\}$, $a_n\lesssim b_n$ or $a_n=O(b_n)$ means $a_n\leq Cb_n$ for some constant $C>0$ independent of $n$. We also write $a_n=o(b_n)$ or $\frac{b_n}{a_n}\rightarrow\infty$ when $\limsup_n\frac{a_n}{b_n}=0$.

\section{Optimality of Iterative Polar Decomposition}\label{sec:upper}
In this section, we derive the upper bound parts of the main results. We first investigate  the $\o(d)$ synchronization  in Section \ref{sec:alg} - Section \ref{sec:od-upper} and then extend the results to the $\so(d)$ synchronization in Section  \ref{sec:sod}.

\subsection{The Algorithm}\label{sec:alg}

For a squared matrix $X\in\mathbb{R}^{d\times d}$ that is of full rank, it admits a singular value decomposition (SVD) $X=UDV^{\T}$ with $U,V\in\mathcal{O}(d)$ and $D$ being diagonal. Then, the polar decomposition of $X$ is given by $X=PQ$, where $P=UV^{\T}$ and $Q=VDV^{\T}$. We denote the first factor, which is called generalized phase, by
\begin{equation}
\mathcal{P}(X)=UV^{\T}. \label{eq:polar-svd}
\end{equation}
It is well known \citep{gower2004procrustes} that $\mathcal{P}(X)$ can also be defined by
\begin{equation}
\mathcal{P}(X) = \argmin_{B\in\mathcal{O}(d)}\fnorm{B-X}^2. \label{eq:procrustes}
\end{equation}
The operator $\mathcal{P}(\cdot)$ satisfies the following properties:
\begin{enumerate}
\item For any $c>0$, $\mathcal{P}(X)=\mathcal{P}(cX)$.
\item For any $R\in\mathcal{O}(d)$, $\mathcal{P}(RX)=R\mathcal{P}(X)$ and $\mathcal{P}(XR^{\T})=\mathcal{P}(X)R^{\T}$.
\item Suppose $X=X^{\T}$ and is positive definite, then $\mathcal{P}(X)=I_d$.
\end{enumerate}

The iterative polar decomposition algorithm for the $\o(d)$ synchronization is given by the following iteration,
\begin{equation}
Z_i^{(t)}=\begin{cases}
\mathcal{P}\left(\sum_{j\in[n]\backslash\{i\}}A_{ij}Y_{ij}Z_j^{(t-1)}\right), & \det\left(\sum_{j\in[n]\backslash\{i\}}A_{ij}Y_{ij}Z_j^{(t-1)}\right)\neq 0, \\
Z_i^{(t-1)}, & \det\left(\sum_{j\in[n]\backslash\{i\}}A_{ij}Y_{ij}Z_j^{(t-1)}\right)=0,
\end{cases}\label{eq:os-iter}
\end{equation}
starting from some initialization $\{Z_i^{(0)}\}_{i\in[n]}$. 
To understand (\ref{eq:os-iter}), we can consider the situation where $\{Z_j^*\}_{j\in[n]\backslash\{i\}}$ are known. Then, to estimate $Z_i^*$, one can apply the MLE that solves the following optimization problem,
$$\min_{Z_i\in\mathcal{O}(d)}\sum_{j\in[n]\backslash\{i\}}A_{ij}\fnorm{Y_{ij}-Z_iZ_j^{*T}}^2.$$
With some straightforward arrangement of the objective function and (\ref{eq:procrustes}), the minimum is achieved by $\mathcal{P}\left(\sum_{j\in[n]\backslash\{i\}}A_{ij}Y_{ij}Z_j^{^*}\right)$ as long as $\sum_{j\in[n]\backslash\{i\}}A_{ij}Y_{ij}Z_j^{^*}$ has full rank. Thus, the iteration (\ref{eq:os-iter}) can be thought of as a local MLE step with the unknown $\{Z_j^*\}_{j\in[n]\backslash\{i\}}$ replaced by $\{Z_j^{(t-1)}\}_{j\in[n]\backslash\{i\}}$ from the last step.

\subsection{An Oracle Perturbation Bound}\label{sec:oracle}

In this section, we give an oracle analysis for (\ref{eq:os-iter})  to understand its statistical properties. An iterative algorithm that is similar to (\ref{eq:os-iter}) has been analyzed by \cite{gao2020exact} in the context of phase synchronization. However, the argument used in \cite{gao2020exact} that leads to the correct constant is limited to the phase synchronization and cannot be used for  the $\o(d)$ or $\so(d)$ synchronization. In the following, we first summarize the analysis  in \cite{gao2020exact} for the phase synchronization and then present our new analysis for the $\o(d)$ synchronization to achieve the correct constant.

As we have mentioned in Section \ref{sec:intro}, phase synchronization is $\so(d)$ synchronization with $d=2$. Since a rotation matrix in $\mathbb{R}^2$ is parametrized by an angle, we can equivalently set up the problem via complex numbers. That is, we have $Y_{ij}=z_i^*\bar{z}_j^*+\sigma W_{ij}\in\mathbb{C}$ for $1\leq i<j\leq n$. Each $z_i^*$ is a complex number with norm $1$ and $\bar{z}_j^*$ stands for the complex conjugate of $z_j^*$. The noise variable $W_{ij}$ is standard complex Gaussian. Suppose we have partial observations on a random graph $\{A_{ij}\}_{1\leq i<j\leq n}$, the generalized power method \citep{boumal2016nonconvex, filbir2020recovery, perry2018message} is given by the following iteration,
\begin{equation}
z_i^{(t)}= \begin{cases}
\frac{\sum_{j\in[n]\backslash\{i\}}A_{ij}Y_{ij}z_j^{(t-1)}}{\left|\sum_{j\in[n]\backslash\{i\}}A_{ij}Y_{ij}z_j^{(t-1)}\right|}, & \left|\sum_{j\in[n]\backslash\{i\}}A_{ij}Y_{ij}z_j^{(t-1)}\right| \neq 0, \\
z_i^{(t-1)}, & \left|\sum_{j\in[n]\backslash\{i\}}A_{ij}Y_{ij}z_j^{(t-1)}\right|=0.
\end{cases} \label{eq:glpm}
\end{equation}
It was shown by \cite{gao2020exact} that (\ref{eq:glpm}) achieves the optimal statistical error with a sharp leading constant after sufficient steps of iterations. The key mathematical ingredient in the analysis of \cite{gao2020exact} is the understanding of a one-step iteration error starting from the truth $z^{(t-1)}=z^*$. That is, we define
\begin{equation}
\check{z}_i=\frac{\sum_{j\in[n]\backslash\{i\}}A_{ij}Y_{ij}z_j^{*}}{\left|\sum_{j\in[n]\backslash\{i\}}A_{ij}Y_{ij}z_j^{*}\right|}, \label{eq:gpm-o}
\end{equation}
and our goal is to give a sharp bound for $|\check{z}_i-z_i^*|^2$. We can easily rearrange the right hand side of (\ref{eq:gpm-o}) as $\check{z}_i=z_i^*\frac{1+e_i}{|1+e_i|}$, with $e_i=\frac{\sigma\sum_{j\in[n]\backslash\{i\}}A_{ij}W_{ij}\bar{z}_i^*z_j^*}{\sum_{j\in[n]\backslash\{i\}}A_{ij}}$. Then,
\begin{equation}
|\check{z}_i-z_i^*|^2 = \left|\frac{1+e_i}{|1+e_i|}-1\right|^2 \leq \frac{|\im(e_i)|^2}{|\re(1+e_i)|^2}, \label{eq:critical-fr}
\end{equation}
where the inequality above is by the fact that
\begin{equation}
\left|\frac{x}{|x|}-1\right|\leq \left|\frac{\im(x)}{\re(x)}\right|\quad\text{ for any }x\in\mathbb{C}\text{ such that }\re(x)>0. \label{eq:luanpao}
\end{equation}
For a proof of (\ref{eq:luanpao}), see Lemma 5.6 in \cite{gao2020exact}. Since it can be shown that $e_i$ is small, the denominator of (\ref{eq:critical-fr}) is close to $1$. The numerical of (\ref{eq:critical-fr}) can be accurately controlled by the Gaussianity of $e_i$ conditioning on the random graph $\{A_{ij}\}_{1\leq i<j\leq n}$. To summarize, the first order behavior of $|\check{z}_i-z_i^*|^2$ is determined by $|\im(e_i)|^2$, which leads to the optimal error of phase synchronization with a sharp constant in \cite{gao2020exact}.

The above analysis relies on (\ref{eq:luanpao}) and critically on the representation of an $\so(2)$ element as a unit complex number. Next, we present our new analysis for the $\o(d)$ synchronization.
To understand the statistical property of (\ref{eq:os-iter}) for the $\o(d)$ synchronization, let us similarly consider an oracle setting with $Z_i^{(t-1)}=Z_i^*$ for all $i\in[n]$. Define
\begin{equation}
\check{Z}_i=\mathcal{P}\left(\sum_{j\in[n]\backslash\{i\}}A_{ij}Y_{ij}Z_j^*\right), \label{eq:and}
\end{equation}
and our goal is to bound $\fnorm{\check{Z}_i-Z_i^*}^2$. Compared with (\ref{eq:gpm-o}), the formula (\ref{eq:and}) does not have a closed form anymore, and the inequality (\ref{eq:luanpao}) that only applies to complex numbers does not have a straightforward extension to general orthogonal matrices.

By the property of $\mathcal{P}(\cdot)$, let us first write write
$$\check{Z}_i=\mathcal{P}\left(\frac{\sum_{j\in[n]\backslash\{i\}}A_{ij}Y_{ij}Z_j^*}{\sum_{j\in[n]\backslash\{i\}}A_{ij}}\right)=\mathcal{P}\left(Z_i^*+E_i\right),$$
where the error matrix is given by
$$E_i=\sigma\frac{\sum_{j\in[n]\backslash\{i\}}A_{ij}W_{ij}Z_j^*}{\sum_{j\in[n]\backslash\{i\}}A_{ij}}.$$
Note that $Z_i^*=\mathcal{P}(Z_i^*)$, and thus bounding $\fnorm{\check{Z}_i-Z_i^*}^2$ requires a perturbation analysis of the operator $\mathcal{P}(\cdot)$, which is given by the following lemma.
\begin{lemma}[Theorem 1 of \cite{li1995new}]\label{lem:li}
Let $X,\wt{X}\in\mathbb{R}^{d\times d}$ be two matrices of full rank. Then,
$$\fnorm{\mathcal{P}(X)-\mathcal{P}(\wt{X})}\leq \frac{2}{s_{\min}(X)+s_{\min}(\wt{X})}\fnorm{X-\wt{X}}.$$
\end{lemma}

By Lemma \ref{lem:li}, we have
\begin{equation}
\fnorm{\check{Z}_i-Z_i^*}^2\leq \left(1-\opnorm{E_i}/2\right)^{-2}\fnorm{E_i}^2. \label{eq:naive-perturb}
\end{equation}
Given $\{A_{ij}\}$, the conditional expectation of $\fnorm{E_i}^2$ is $\frac{\sigma^2d^2}{\sum_{j\in[n]\backslash\{i\}}A_{ij}}$. Moreover, $\sum_{j\in[n]\backslash\{i\}}A_{ij}$ concentrates around $(n-1)p$. Therefore, it can be shown that
\begin{equation}
\fnorm{E_i}^2=(1+o_{\mathbb{P}}(1))\frac{\sigma^2d^2}{np}, \label{eq:ppppp}
\end{equation}
under appropriate conditions. When $\frac{\sigma^2d^2}{np}=o(1)$, we have the bound
\begin{equation}
\fnorm{\check{Z}_i-Z_i^*}^2\leq (1+o_{\mathbb{P}}(1))\frac{\sigma^2d^2}{np}. \label{eq:oracle-subopt}
\end{equation}

However, compared with our minimax lower bound in Theorem \ref{thm:lower-od}, it is clear that the error bound $\frac{\sigma^2d^2}{np}$ is not optimal. This is due to a naive application of Lemma \ref{lem:li}. Below, we present an improved analysis of $\fnorm{\check{Z}_i-Z_i^*}^2$. Recall the definition of $E_i$ and the properties of $\mathcal{P}(\cdot)$, and we have
$$\fnorm{\check{Z}_i-Z_i^*}^2 = \fnorm{\mathcal{P}(I_d+E_iZ_i^{*\T})-I_d}^2.$$
The key observation is that we can write $I_d$ as the operator $\mathcal{P}(\cdot)$ applied to any positive definite matrix. With this idea, we have
\begin{equation}
\fnorm{\check{Z}_i-Z_i^*}^2 = \fnorm{\mathcal{P}(I_d+E_iZ_i^{*\T})-\mathcal{P}\left(\text{any positive definite matrix}\right)}^2. \label{eq:anderson}
\end{equation}
We shall choose a positive definite matrix whose difference from $I_d+E_iZ_i^*$ is as small as possible and then apply Lemma \ref{lem:li}. It turns out a correct choice is $I_d+\frac{1}{2}E_iZ_i^*+\frac{1}{2}Z_i^{*\T}E_i^{\T}$. When $\frac{\sigma^2d^2}{np}=o(1)$, we can view $I_d+\frac{1}{2}E_iZ_i^*+\frac{1}{2}Z_i^{*\T}E_i^{\T}$ as a small perturbation from $I_d$ by (\ref{eq:ppppp}), and thus it is positive definite.
Apply Lemma \ref{lem:li}, and we have
\begin{eqnarray*}
\fnorm{\check{Z}_i-Z_i^*}^2 &=& \fnorm{\mathcal{P}(I_d+E_iZ_i^{*\T})-\mathcal{P}\left(I_d+\frac{1}{2}E_iZ_i^*+\frac{1}{2}Z_i^{*\T}E_i^{\T}\right)}^2 \\
&\leq& \left(1-\opnorm{E_i}\right)^{-2}\fnorm{\frac{1}{2}E_iZ_i^{*\T}-\frac{1}{2}Z_i^*E_i^{\T}}^2.
\end{eqnarray*}
Compared with the previous bound (\ref{eq:naive-perturb}), the $\fnorm{E_i}^2$ in (\ref{eq:naive-perturb}) has been improved to $\fnorm{\frac{1}{2}E_iZ_i^{*\T}-\frac{1}{2}Z_i^*E_i^{\T}}^2$. To see why this is an improvement, we have two simple observations:
\begin{enumerate}
\item The diagonal entries of $\frac{1}{2}E_iZ_i^{*\T}-\frac{1}{2}Z_i^*E_i^{\T}$ are all zero, whereas those of $E_i$ are all nonzero.
\item The off-diagonal entries of $\frac{1}{2}E_iZ_i^{*\T}-\frac{1}{2}Z_i^*E_i^{\T}$ have smaller variance. For any $1\leq a<b\leq d$, we have $\left(\frac{1}{2}E_iZ_i^{*\T}-\frac{1}{2}Z_i^*E_i^{\T}\right)_{ab}|A\sim \mathcal{N}\left(0,\frac{\sigma^2}{2\sum_{j\in[n]\backslash\{i\}}A_{ij}}\right)$, compared with $(E_i)_{ab}|A\sim \mathcal{N}\left(0,\frac{\sigma^2}{\sum_{j\in[n]\backslash\{i\}}A_{ij}}\right)$.
\end{enumerate}
By direct calculation, the conditional expectation of $\fnorm{\frac{1}{2}E_iZ_i^{*\T}-\frac{1}{2}Z_i^*E_i^{\T}}^2$ given $\{A_{ij}\}$ is $\frac{\sigma^2d(d-1)}{2\sum_{j\in[n]\backslash\{i\}}A_{ij}}$, and it can be shown that
$$\fnorm{\frac{1}{2}E_iZ_i^{*\T}-\frac{1}{2}Z_i^*E_i^{\T}}^2=(1+o_{\mathbb{P}}(1))\frac{\sigma^2d(d-1)}{2np},$$
under appropriate conditions. This leads to the bound
\begin{equation}
\fnorm{\check{Z}_i-Z_i^*}^2\leq (1+o_{\mathbb{P}}(1))\frac{\sigma^2d(d-1)}{2np}, \label{eq:oracle-opt}
\end{equation}
which is optimal in view of the minimax lower bound in Theorem \ref{thm:lower-od}. Compared with $d^2$ in (\ref{eq:oracle-subopt}), the factor $\frac{1}{2}d(d-1)$ in (\ref{eq:oracle-opt}) is the correct degrees of freedom of the space $\mathcal{O}(d)$.

\subsection{Analysis of the Iteration}\label{sec:ana-ite}

Having understood how a one-step iteration (\ref{eq:os-iter}) would achieve the optimal statistical error bound if it were started from the truth, we are ready to analyze the evolution of (\ref{eq:os-iter}) starting from an initialization that is close to the truth. Let us first shorthand the formula (\ref{eq:os-iter}) by
$$Z^{(t)}=f(Z^{(t-1)}).$$
In other words, we have introduced a map $f:\o(d)^n\rightarrow \o(d)^n$ such that $f(Z^{(t-1)})_i$ is defined by (\ref{eq:os-iter}). We characterize the evolution of the loss function (\ref{eq:loss-od-def}) through the map $f$ by the following lemma.
\begin{lemma}\label{lem:main}
For the $\o(d)$ Synchronization (\ref{eq:od-setting}),
assume $\frac{np}{\sigma^2}\geq c_1$ and $\frac{np}{\log n}\geq c_2$ for some sufficiently large constants $c_1,c_2>0$ and $2\leq d\leq C$ for some constant $C>0$. Then, for any $\gamma\in[0,1/16)$, we have
\begin{eqnarray*}
&& \mathbb{P}\left(\ell(f(Z),Z^*)\leq \delta_1\ell(Z,Z^*)+(1+\delta_2)\frac{\sigma^2d(d-1)}{2np}\text{ for all }Z\in\o(d)^n\text{ such that }\ell(Z,Z^*)\leq\gamma\right) \\
&\geq& 1-n^{-9}-\exp\left(-\br{\frac{ np}{\sigma^2}}^{1/4}\right),
\end{eqnarray*}
where $\delta_1=C_1\sqrt{\frac{\log n+\sigma^2}{np}}$ and $\delta_2=C_2\left(\gamma^2+\frac{\log n+\sigma^2}{np}\right)^{1/4}$ for some constants $C_1,C_2>0$ that only depend on $C$.
\end{lemma}

To understand the consequence of Lemma \ref{lem:main}, we can first do a sanity check by setting $Z=Z^*$. This results in the bound
$$\ell(f(Z^*),Z^*)\leq (1+\delta_2)\frac{\sigma^2d(d-1)}{2np},$$
which is the one-step iteration error starting from the truth, and thus the oracle analysis in Section \ref{sec:oracle} is recovered.

More generally, as long as $Z^{(t-1)}$ satisfies $\ell(Z^{(t-1)},Z^*)\leq\gamma$ for some $\gamma\in(0,1/16)$, we have
\begin{equation}
\ell(Z^{(t)},Z^*) \leq \delta_1\ell(Z^{(t-1)},Z^*)+(1+\delta_2)\frac{\sigma^2d(d-1)}{2np}. \label{eq:iteration-bound}
\end{equation}
From (\ref{eq:iteration-bound}), we know that $\ell(Z^{(t)},Z^*)\leq \delta_1\gamma+(1+\delta_2)\frac{\sigma^2d(d-1)}{2np}$, which can again be bounded by $\gamma$ under the condition that $\frac{np}{\sigma^2}\geq c_1$ and $\frac{np}{\log n}\geq c_2$ for some sufficiently large constants $c_1,c_2>0$. By mathematical induction, we can conclude that (\ref{eq:iteration-bound}) holds for all $t\geq 1$ as long as $\ell(Z^{(0)},Z^*)\leq\gamma$. We can also rearrange the one-step iteration bound (\ref{eq:iteration-bound}) into a linear convergence result,
\begin{equation}
\ell(Z^{(t)},Z^*) \leq \delta_1^t\ell(Z^{(0)},Z^*)+\frac{1+\delta_2}{1-\delta_1}\frac{\sigma^2d(d-1)}{2np}, \label{eq:linear}
\end{equation}
for all $t\geq 1$. The bound (\ref{eq:linear}) implies that $Z^{(t)}$ will eventually be statistically optimal after sufficient steps of iterations. However, it does not imply that $Z^{(t)}$ converges to any fixed object. As is shown by \cite{zhong2018near,ling2020improved} via a leave-one-out argument, the algorithmic convergence of the iterative algorithm requires the condition $\sigma^2=O\left(\frac{n}{\log n}\right)$ at least when $p=1$. In comparison, the bound (\ref{eq:linear}) guarantees the statistical optimality without characterizing its convergence property, and thus only requires $\frac{np}{\sigma^2}$ to be sufficiently large.

As we have discussed in the introduction section, the two conditions $\frac{np}{\sigma^2}\geq c_1$ and $\frac{np}{\log n}\geq c_2$ are essentially necessary for the result to hold, at least when $d$ is bounded by a constant. If both conditions are slightly strengthened to $\frac{np}{\sigma^2}\rightarrow\infty$ and $\frac{np}{\log n}\rightarrow\infty$, the same result of Lemma \ref{lem:main} will hold with vanishing $\delta_1$ and $\delta_2$.

The proof of Lemma \ref{lem:main} also holds more generally for $d$ that can potentially grow. Without assuming $d\leq C$ for some constant $C>0$, we would obtain the same high probability bound with $\delta_1=C_1\sqrt{\frac{d\log n+d^2\sigma^2}{np}}$, $\delta_2=C_2\left(\gamma^2+\frac{\log n+d\sigma^2}{np}\right)^{1/4}$, and the conditions replaced by $\frac{np}{d^2\sigma^2}\geq c_1$ and $\frac{np}{d\log n}\geq c_2$.

\subsection{Optimal Upper Bound} \label{sec:od-upper}

To derive the minimax upper bound for the $\o(d)$ synchronization from Lemma \ref{lem:main}, we need to construct an initialization $Z^{(0)}\in\o(d)^n$ whose statistical error $\ell(Z^{(0)},Z^*)$ is sufficiently small. Let us first organize the observations $\{Y_{ij}\}_{1\leq i<j\leq n}$ into a matrix $Y\in\mathbb{R}^{nd\times nd}$. That is,
\begin{equation}
Y=\begin{pmatrix}
Y_{11} & \cdots & Y_{1n} \\
\vdots & \ddots & \vdots \\
Y_{n1} & \cdots & Y_{nn}
\end{pmatrix}, \label{eq:big-matrix}
\end{equation}
where $Y_{ji}=Y_{ij}$ for all $1\leq i<j\leq n$ and $Y_{ii}=I_d$ for all $i\in[n]$. The noise matrices $\{W_{ij}\}_{1\leq i<j\leq n}$ can be organized into $W\in\mathbb{R}^{nd\times nd}$ with the same arrangement as (\ref{eq:big-matrix}), and we set $W_{ji}=W_{ij}$ for all $1\leq i<j\leq n$ and $W_{ii}=0$ for all $i\in[n]$. Then, we can write the model (\ref{eq:od-model}) as
$$Y=Z^*Z^{*\T}+\sigma W,$$
where $Z^{*\T}=(Z_1^{*\T},\cdots,Z_n^{*\T})\in\mathbb{R}^{d\times nd}$. In other words, $Y$ can be viewed as a noisy version of the rank-$d$ matrix $Z^*Z^{*\T}$, and thus we can use a spectral method to estimate the column space of $Z^*$. Since we do not observe all $Y_{ij}$'s, the spectral method can be applied to $(A\otimes \mathds{1}_d\mathds{1}_d^T)\circ Y$, where $A\in\{0,1\}^{n\times n}$ with $A_{ji}=A_{ij}$ for all $1\leq i<j\leq n$ and $A_{ii}=0$ for all $i\in[n]$. Recall that $\otimes$ stands for the matrix Kronecker product and $\circ$ denotes the Hadamard product. To compute $Z^{(0)}\in\o(d)^n$, we first find
\begin{equation}
\wh{U}=\argmax_{U\in\o(nd,d)}\Tr\left(U^{\T}\left((A\otimes \mathds{1}_d\mathds{1}_d^T)\circ Y\right)U\right), \label{eq:PCA}
\end{equation}
and then compute
\begin{equation}
Z_i^{(0)}=\begin{cases}
\P(\wh{U}_i), & \det(\wh{U}_i)\neq 0, \\
I_d, & \det(\wh{U}_i)=0,
\end{cases} \label{eq:od-ini}
\end{equation}
for all $i\in[n]$. Here, $\wh{U}_i$ stands for the $i$th $d\times d$ block of $\wh{U}$ and thus $\wh{U}^{\T}=(\wh{U}_1^{\T},\cdots,\wh{U}_n^{\T})$. The error bound of $Z^{(0)}$ is given by the following lemma.
\begin{lemma}\label{lem:od-ini}
For the $\o(d)$ Synchronization (\ref{eq:od-setting}),
assume $\frac{np}{\log n}\geq c$ for some sufficiently large constants $c>0$ and $2\leq d\leq C_1$ for some constant $C_1>0$. Then, we have
$$\ell(Z^{(0)},Z^*)\leq C\frac{\sigma^2+1}{np},$$
with probability at least $1-n^{-9}$ for some constant $C>0$ only depending on $C_1$.
\end{lemma}

The error rate $\frac{\sigma^2+1}{np}$ is the sum of two terms. The first term $\frac{\sigma^2}{np}$ is from the additive Gaussian noise in the model (\ref{eq:od-model}).  The second term $\frac{1}{np}$ is a consequence of the randomness from the graph. It comes from an upper bound $\opnorm{A-\mathbb{E}A}\lesssim \sqrt{np}$. In fact, it can be slightly improved by $\opnorm{A-\mathbb{E}A}\lesssim \sqrt{n(p\wedge (1-p))}$, which makes a difference when $1-p$ is small. This leads to the second term $\frac{1}{np}$ replaced by $\frac{p\wedge(1-p)}{np^2}$. As a result, when $\sigma^2=0$ and $p=1$, we have $\ell(Z^{(0)},Z^*)=0$, i.e., perfect recovery of $Z^*$.

By Lemma \ref{lem:od-ini}, we know that $\ell(Z^{(0)},Z^*)$ is sufficiently small when $\frac{np}{\sigma^2}$ and $\frac{np}{\log n}$ are sufficiently large. Then, we can directly apply Lemma \ref{lem:main} and its consequence  (\ref{eq:linear}) to obtain the desired upper bound.

\begin{thm}\label{thm:upper-od}
For the $\o(d)$ Synchronization (\ref{eq:od-setting}), assume $\frac{np}{\sigma^2}\geq c_1$ and $\frac{np}{\log n}\geq c_2$ for some sufficiently large constants $c_1,c_2>0$ and $2\leq d\leq C_1$ for some constant $C_1>0$. Consider the algorithm (\ref{eq:os-iter}) initialized by (\ref{eq:od-ini}). We have
$$\ell(Z^{(t)},Z^*)\leq \left(1+C\left(\frac{\log n+\sigma^2}{np}\right)^{1/4}\right)\frac{\sigma^2d(d-1)}{2np},$$
for all $t\geq\log\left(\frac{1}{\sigma^2}\right)$ with probability at least $1-2n^{-9}-\exp\left(-\br{\frac{ np}{\sigma^2}}^{1/4}\right)$ for some constant $C>0$ only depending on $C_1$.
\end{thm}

\begin{remark}\label{rmk:1}
The proof of Lemma \ref{lem:od-ini} also holds more generally for $d$ that can potentially grow.
When $d$ grows, the initialization has the error bound $\ell(Z^{(0)},Z^*)\leq C\frac{d^4(d\sigma^2+1)}{np}$ with high probability. As a consequence, without assuming $d\leq C_1$, the result of Theorem \ref{thm:upper-od} becomes
$$\ell(Z^{(t)},Z^*)\leq \left(1+C\left(\left(\frac{d\log n+d^2\sigma^2}{np}\right)^{1/4}+\frac{d^2}{\sqrt{np}}\right)\right)\frac{\sigma^2d(d-1)}{2np},$$
with high probability under the conditions that $\frac{np}{d\log n}$, $\frac{np}{d^2\sigma^2}$ and $\frac{np}{d^4}$ are sufficiently large. The same iterative algorithm has also been analyzed by \cite{ling2020improved} for an $\ell_{\infty}$ type loss when $p=1$, and they showed that
$$
\min_{B\in\o(d)}\max_{i\in[n]}\fnorm{{Z}_i^{(t)}-Z_i^*B}^2 \leq C\frac{\sigma^2(d^2+d\log n)}{n}, \label{eq:ling}
$$
with high probability under the condition that $\frac{n}{\sigma^2(d^2+d\log n)}$ is sufficiently large.
\end{remark}

\subsection{Rotation Group Synchronization}\label{sec:sod}

In this section, we study $\so(d)$ synchronization, where our goal is to estimate $Z^*\in\so(d)^n$ from the observations (\ref{eq:od-model}) on a random graph. The loss function for this problem is defined by
$$\bar{\ell}(Z,Z^*)=\min_{B\in\so(d)}\frac{1}{n}\sum_{i=1}^n\fnorm{Z_i-Z_i^*B}^2,$$
for any $Z,Z^*\in\so(d)^n$.
The $\so(d)$ synchronization problem requires a slight modification of the iterative algorithm (\ref{eq:os-iter}). To do this, we first introduce an operator $\bar{\P}(\cdot)$ that maps a $d\times d$ full-rank matrix to $\so(d)$. For a full-rank squared matrix $X\in\mathbb{R}^{d\times d}$ that admits an SVD $X=UDV^{\T}$, we define
$$\bar{\mathcal{P}}(X)=U\begin{pmatrix}
I_{d-1} & 0 \\
0 & \det(UV^{\T})
\end{pmatrix}V^{\T}.$$
The only difference from $\P(X)$ is the diagonal matrix with the last entry $\det(UV^{\T})$ sandwiched between $U$ and $V$. It is clear that $\det(UV^{\T})\in\{-1,1\}$ and thus $\det(\bar{\mathcal{P}}(X))=1$, which implies $\bar{\mathcal{P}}(X)\in\so(d)$. By \cite{kabsch1978discussion}, $\bar{\mathcal{P}}(X)$ can also be characterized as the solution to an optimization problem. That is,
$$\bar{\mathcal{P}}(X)=\argmin_{B\in\so(d)}\fnorm{B-X}^2.$$
Then, similar to the motivation behind the iteration (\ref{eq:os-iter}), we consider an iterative procedure for $\so(d)$ synchronization,
\begin{equation}
Z_i^{(t)}=\begin{cases}
\bar{\mathcal{P}}\left(\sum_{j\in[n]\backslash\{i\}}A_{ij}Y_{ij}Z_j^{(t-1)}\right), & \det\left(\sum_{j\in[n]\backslash\{i\}}A_{ij}Y_{ij}Z_j^{(t-1)}\right)\neq 0, \\
Z_i^{(t-1)}, & \det\left(\sum_{j\in[n]\backslash\{i\}}A_{ij}Y_{ij}Z_j^{(t-1)}\right)=0.
\end{cases}\label{eq:sos-iter}
\end{equation}

The iteration (\ref{eq:sos-iter}) enjoys a similar convergence property as given by Lemma \ref{lem:main} with a good initialization. This result is stated as Lemma \ref{lem:main-2} in Section \ref{sec:pf-sod}. Mathematically, one can show that as long as $\bar{\ell}(Z^{(t-1)},Z^*)\leq\gamma$ for some sufficiently small $\gamma$, the determinant of $\sum_{j\in[n]\backslash\{i\}}A_{ij}Y_{ij}Z_j^{(t-1)}$ is positive for most $i\in[n]$ so that $\bar{\P}(\cdot)=\P(\cdot)$ for those $i$'s. With this argument, the same proof that leads to the conclusion of Lemma \ref{lem:main} also characterizes the convergence of (\ref{eq:sos-iter}).

To initialize the iterative procedure (\ref{eq:sos-iter}), we also use a spectral method. Define
\begin{equation}
Z_i^{(0)}=\begin{cases}
\bar{\P}(\wh{U}_i), & \det(\wh{U}_i)\neq 0, \\
I_d, & \det(\wh{U}_i)=0,
\end{cases} \label{eq:sod-ini}
\end{equation}
where $\wh{U}^{\T}=(\wh{U}_1^{\T},\cdots,\wh{U}_n^{\T})$ is given by (\ref{eq:PCA}). It is clear that $Z_i^{(0)}\in\so(d)$ for all $i\in[n]$. We show in Lemma \ref{lem:sod-ini} that $\bar{\ell}(Z^{(0)},Z^*)$ is sufficiently small. Together with the statistical property of the iterative procedure (\ref{eq:sos-iter}), we have the following upper bound result for $\so(d)$ synchronization.

\begin{thm}\label{thm:upper-sod}
For the $\so(d)$ Synchronization (\ref{eq:sod-setting}), assume $\frac{np}{\sigma^2}\geq c_1$ and $\frac{np}{\log n}\geq c_2$ for some sufficiently large constants $c_1,c_2>0$ and $2\leq d\leq C_1$ for some constant $C_1>0$. Consider the algorithm (\ref{eq:sos-iter}) initialized by (\ref{eq:sod-ini}). We have
$$\bar{\ell}(Z^{(t)},Z^*)\leq \left(1+C\left(\frac{\log n+\sigma^2}{np}\right)^{1/4}\right)\frac{\sigma^2d(d-1)}{2np},$$
for all $t\geq\log\left(\frac{1}{\sigma^2}\right)$ with probability at least $1-2n^{-9}-\exp\left(-\br{\frac{ np}{\sigma^2}}^{1/4}\right)$ for some constant $C>0$ only depending on $C_1$.
\end{thm}

\section{Minimax Lower Bound}\label{sec:lower}

In this section, we derive the lower bound part of the main result. We first consider $\o(d)$ synchronization, and the minimax risk is given by
\begin{equation}
\inf_{\wh{Z}\in\o(d)^n}\sup_{Z\in\o(d)^n}\mathbb{E}_Z\ell(\wh{Z},Z)=\inf_{\wh{Z}\in\o(d)^n}\sup_{Z\in\o(d)^n}\mathbb{E}_Z\left[\min_{B\in\o(d)}\frac{1}{n}\sum_{i=1}^n\fnorm{\wh{Z}_i-Z_iB}^2\right].
\end{equation}
Our first step is to apply Lemma \ref{lem:loss-re} and lower bound the loss function $\ell(\wh{Z},Z)$ by
\begin{equation}
\ell(\wh{Z},Z) \geq \frac{1}{2n^2}\sum_{i=1}^n\sum_{j=1}^n\fnorm{\wh{Z}_i\wh{Z}_j^{\T}-Z_iZ_j^{\T}}^2. \label{eq:lower-step1}
\end{equation}
Compared with $\ell(\wh{Z},Z)$, the right hand side of (\ref{eq:lower-step1}) is a loss that is decomposable across all pairs $(i,j)\in[n]^2$, and therefore it is sufficient to lower bound the 
estimation error of each individual $Z_iZ_j^{\T}$ and then to aggregate the results. Following this strategy, we have
\begin{eqnarray}
\nonumber && \inf_{\wh{Z}\in\o(d)^n}\sup_{Z\in\o(d)^n}\mathbb{E}_Z\ell(\wh{Z},Z) \\
\nonumber &\geq& \frac{1}{2n^2}\inf_{\wh{Z}\in\o(d)^n}\sup_{Z\in\o(d)^n}\sum_{i=1}^n\sum_{j=1}^n\mathbb{E}_Z\fnorm{\wh{Z}_i\wh{Z}_j^{\T}-Z_iZ_j^{\T}}^2 \\
\label{eq:lower-ineq1} &\geq& \frac{1}{2n^2}\inf_{\wh{Z}\in\o(d)^n}\sum_{1\leq i\neq j\leq n}\int\mathbb{E}_Z\fnorm{\wh{Z}_i\wh{Z}_j^{\T}-Z_iZ_j^{\T}}^2\prod_{k=1}^n\diff \Pi(Z_k) \\
\label{eq:lower-ineq2} &\geq& \frac{1}{2n^2}\sum_{1\leq i\neq j\leq n}\int\left(\inf_{\wh{T}}\int\int \mathbb{E}_Z\fnorm{\wh{T}-Z_iZ_j^{\T}}^2\diff \Pi(Z_i)\diff \Pi(Z_j)\right)\prod_{k\in[n]\backslash\{i,j\}}\diff \Pi(Z_k),
\end{eqnarray}
where $\Pi$ is some probability distribution supported on $\o(d)$ to be specified later. The inequality (\ref{eq:lower-ineq1}) lower bounds the minimax risk with a Bayes risk, and (\ref{eq:lower-ineq2}) is lower bounding infimum of average by average of infimum. Now it suffices to lower bound
\begin{equation}
\inf_{\wh{T}}\int\int \mathbb{E}_Z\fnorm{\wh{T}-Z_iZ_j^{\T}}^2\diff \Pi(Z_i)\diff \Pi(Z_j), \label{eq:local-problem}
\end{equation}
for each $\{Z_k\}_{k\in[n]\backslash\{i,j\}}$ and each $i\neq j$. The quantity (\ref{eq:local-problem}) can be understood as the Bayes risk of estimating $Z_iZ_j^{\T}$ given the knowledge of $\{Z_k\}_{k\in[n]\backslash\{i,j\}}$.

To analyze (\ref{eq:local-problem}), we first need to construct a probability distribution $\Pi$ on $\o(d)$. Given $\{r_{ab}\}_{1\leq a<b\leq d}$ and $\{s_{ab}\}_{1\leq b\leq a\leq d}$, we can form the following $d\times d$ matrix,
\begin{equation}
Q=\begin{pmatrix}
s_{11} & r_{12} & r_{13} & \cdots & r_{1d} \\
s_{21} & s_{22} & r_{23} & \cdots & r_{2d} \\
\vdots & \vdots & \vdots & \ddots  & \vdots \\
s_{d-1,1} & s_{d-1,2} & s_{d-1,3} & \cdots & r_{d-1,d} \\
s_{d1} & s_{d2} & s_{d3} & \cdots & s_{dd}
\end{pmatrix}.\label{eq:Q-prior}
\end{equation}
In other words, $\{r_{ab}\}_{1\leq a<b\leq d}$ are the upper triangular elements of $Q$, and the lower triangular and the diagonal elements of $Q$ are given by $\{s_{ab}\}_{1\leq b\leq a\leq d}$. We are going to specify the values of $\{s_{ab}\}_{1\leq b\leq a\leq d}$ by $\{r_{ab}\}_{1\leq a<b\leq d}$ so that $Q$ is a matrix of degrees of freedom $\frac{1}{2}d(d-1)$. To do this, let us introduce some new notation. For each integer $a\geq 2$, $Q_{a-1}\in\mathbb{R}^{(a-1)\times (a-1)}$ is the submatrix of $Q$ collecting the first $a-1$ rows and columns. We also define
$$S_{a-1}=\begin{pmatrix}
s_{a1} \\
\vdots \\
s_{a,a-1}
\end{pmatrix},\quad R_{a-1}=\begin{pmatrix}
r_{1a} \\
\vdots \\
r_{a-1,a}
\end{pmatrix}\quad\text{and}\quad v_{a-1}=\begin{pmatrix}
r_{1,a+1}r_{a,a+1} + \cdots + r_{1d}r_{ad} \\
\vdots \\
r_{a-1,a+1}r_{a,a+1} + \cdots + r_{a-1,d}r_{ad}
\end{pmatrix}.$$
All the three vectors above belong to $\mathbb{R}^{a-1}$. The construction of $\{s_{ab}\}_{1\leq b\leq a\leq d}$ is given by the following procedure.
\begin{enumerate}
\item We first set $s_{11}=\sqrt{1-\left(r_{12}^2+\cdots+r_{1d}^2\right)}$.
\item Given the values of $\{s_{ab}\}_{1\leq b\leq a\leq k-1}$, we find $s_{k1}, s_{k2},\cdots, s_{kk}$ through the equations
\begin{eqnarray}
\label{eq:eq1} Q_{k-1}S_{k-1} + s_{kk}R_{k-1} &=& -v_{k-1} \\
\label{eq:eq2} \|S_{k-1}\|^2 + s_{kk}^2 &=& 1-\left(r_{k,k+1}^2+\cdots+r_{kd}^2\right).
\end{eqnarray}
Note that the equations above have two sets of real solutions (under the assumption of Lemma \ref{lem:Q-exist}). We take the set of solutions with the larger value of $s_{kk}$.
\end{enumerate}
After going through the above procedure, the matrix $Q$ in the form of (\ref{eq:Q-prior}) is fully determined by its upper triangular elements $\{r_{ab}\}_{1\leq a<b\leq d}$, and therefore we can write $Q=Q(r)$. The equation (\ref{eq:eq1}) guarantees that the rows of $Q(r)$ are orthogonal to each other and (\ref{eq:eq2}) implies that each row of $Q(r)$ is a unit vector. The following lemma characterizes a sufficient condition on $\{r_{ab}\}_{1\leq a<b\leq d}$ that implies the existence of $Q(r)$.
\begin{lemma}\label{lem:Q-exist}
Assume $\max_{1\leq a<b\leq d}|r_{ab}|\leq \frac{1}{8d^{5/2}}$. Then, $Q(r)$ is well defined and the following properties are satisfied:
\begin{enumerate}
\item $Q(r)\in\so(d)$;
\item $\max_{1\leq b<a\leq d}|s_{ab}|\leq \frac{1}{4d^2}$ and $\min_{a\in[d]}s_{aa}\geq\frac{7}{8}$;
\item $\max_{1\leq a<b\leq d}\max_{u\in[d]}\sqrt{\sum_{v=1}^u\left|\frac{\partial s_{uv}}{\partial r_{ab}}\right|^2}\leq 5$.
\end{enumerate} 
\end{lemma}
According to Lemma \ref{lem:Q-exist}, the constructed $Q(r)$ can also be used for deriving the minimax lower bound of $\so(d)$ synchronization. Moreover, its entries and the derivatives with respect to $\{r_{ab}\}_{1\leq a<b\leq d}$ are well controlled, which means that the matrix $Q(r)$ is smoothly parametrized by $r$. Let $P$ be a distribution under which $r_{ab}\sim \mu$ independently for all $1\leq a<b\leq d$ with some $\mu$ being a smooth probability density function supported on $\left[-\frac{1}{8d^{5/2}},\frac{1}{8d^{5/2}}\right]$. To be specific, we can set $\mu(t)\propto \exp\left(-\frac{1}{1-64d^5t^2}\right)\mathbb{I}\{|t|\leq 1/(8d^{5/2})\}$. Then, by letting $\Pi$ be the induced probability measure of $Q(r)$ with $\{r_{ab}\}_{1\leq a<b\leq d}\sim P$, we can write (\ref{eq:local-problem}) as
\begin{equation}
\inf_{\wh{T}}\int\int \mathbb{E}_Z\fnorm{\wh{T}-Z_i(r)Z_j(r')^{\T}}^2\diff P(r)\diff P(r'), \label{eq:local-problem-induced}
\end{equation}
where we have some slight abuse of notation that $Z_i=Z_i(r)=Q(r)$ and $Z_j=Z_j(r')=Q(r')$. Compared with (\ref{eq:local-problem}), the distribution $P$ in (\ref{eq:local-problem-induced}) is a standard probability measure on $\mathbb{R}^{\frac{d(d-1)}{2}}$. The Bayes risk (\ref{eq:local-problem-induced}) can then be lower bounded via van Trees' inequality \citep{gill1995applications}.
\begin{lemma}\label{lem:trees}
Assume $\frac{np}{\sigma^2}\geq c$ for some sufficiently large constants $c>0$ and $2\leq d\leq C_1$ for some constant $C_1>0$. Then, we have
$$\inf_{\wh{T}}\int\int \mathbb{E}_Z\fnorm{\wh{T}-Z_i(r)Z_j(r')^{\T}}^2\diff P(r)\diff P(r')\geq \left(1-C\left(\frac{1}{n}+\frac{\sigma^2}{np}\right)\right)\frac{\sigma^2d(d-1)}{np},$$
for some constant $C>0$ only depending on $C_1$.
\end{lemma}
The proof of Lemma \ref{lem:trees}, which verifies the technical conditions of \cite{gill1995applications}, is given in Section \ref{sec:pf-lower}. Note that these technical conditions are implied by the conclusion of Lemma \ref{lem:Q-exist}. In view of the inequality (\ref{eq:lower-ineq2}), we immediately have the following theorem.

\begin{thm}\label{thm:lower-od}
Assume $\frac{np}{\sigma^2}\geq c$ for some sufficiently large constants $c>0$ and $2\leq d\leq C_1$ for some constant $C_1>0$. Then, we have
$$\inf_{\wh{Z}\in\o(d)^n}\sup_{Z\in\o(d)^n}\mathbb{E}_Z\ell(\wh{Z},Z)\geq \left(1-C\left(\frac{1}{n}+\frac{\sigma^2}{np}\right)\right)\frac{\sigma^2d(d-1)}{2np},$$
for some constant $C>0$ only depending on $C_1$.
\end{thm}

The derivation of the minimax lower bound for $\so(d)$ synchronization follows the same approach. According to Lemma \ref{lem:loss-re}, the inequality (\ref{eq:lower-step1}) also holds for the loss $\bar{\ell}(\wh{Z},Z)$. Moreover, by Lemma \ref{lem:Q-exist}, the construction of $Q(r)$ is already in $\so(d)$. Then, an analogous result to Lemma \ref{lem:trees} also holds with some straightforward modification. This leads to the following theorem.

\begin{thm}\label{thm:lower-sod}
Assume $\frac{np}{\sigma^2}\geq c$ for some sufficiently large constants $c>0$ and $2\leq d\leq C_1$ for some constant $C_1>0$. Then, we have
$$\inf_{\wh{Z}\in\so(d)^n}\sup_{Z\in\so(d)^n}\mathbb{E}_Z\bar{\ell}(\wh{Z},Z)\geq \left(1-C\left(\frac{1}{n}+\frac{\sigma^2}{np}\right)\right)\frac{\sigma^2d(d-1)}{2np},$$
for some constant $C>0$ only depending on $C_1$.
\end{thm}

\section{Proofs}\label{sec:pf}

\subsection{Some Auxiliary Lemmas}

\begin{lemma}\label{lem:ER-graph}
Assume $\frac{np}{\log n}>c$ for some sufficiently large constant $c>0$. Then, we have
$$\max_{i\in[n]}\left(\sum_{j\in[n]\backslash\{i\}}(A_{ij}-p)\right)^2\leq Cnp\log n,$$
and
$$\opnorm{A-\mathbb{E}A}\leq C\sqrt{np},$$
with probability at least $1-n^{-10}$ for some constant $C>0$.
\end{lemma}
\begin{proof}
The first result is a direct application of union bound and Bernstein's inequality. The second result is Theorem 5.2 of \cite{lei2015consistency}. 
\end{proof}

\begin{lemma}\label{lem:bandeira}
Assume $\frac{np}{\log n}>c$ for some sufficiently large constant $c>0$. Then, we have
$$\opnorm{(A\otimes \mathds{1}_d\mathds{1}_d^T)\circ W}\leq C\sqrt{dnp},$$
with probability at least $1-n^{-10}$ for some constant $C>0$.
\end{lemma}
\begin{proof}
We use $\mathbb{P}_A$ for the conditional probability $\mathbb{P}(\cdot|A)$. Define the event
$$\mathcal{A}=\left\{\max_{i\in[n]}\sum_{j\in[n]\backslash\{i\}}A_{ij}\leq 2np\right\}.$$
Under the assumption $\frac{np}{\log n}>c$, we have $\mathbb{P}(\mathcal{A}^c)\leq n^{-11}$ by Bernstein's inequality and a union bound argument. It is clear that the largest row $\ell_2$ norm of $A\otimes \mathds{1}_d\mathds{1}_d^T$ is bounded by $\sqrt{2dnp}$ under the event $\mathcal{A}$.
By Corollary 3.9 of \cite{bandeira2016sharp}, we have
$$\sup_{A\in\mathcal{A}}\mathbb{P}_A\left(\opnorm{(A\otimes \mathds{1}_d\mathds{1}_d^T)\circ W} > C_1\sqrt{dnp} + t\right)\leq e^{-t^2/4},$$
for some constant $C_1>0$. This implies that $\sup_{A\in\mathcal{A}}\mathbb{P}_A\left(\opnorm{(A\otimes \mathds{1}_d\mathds{1}_d^T)\circ W} > C_2\sqrt{dnp}\right)\leq n^{-11}$ for some constant $C_2>0$. Thus, we have
\begin{eqnarray*}
&& \mathbb{P}\left(\opnorm{(A\otimes \mathds{1}_d\mathds{1}_d^T)\circ W} > C_2\sqrt{dnp}\right) \\
&\leq& \mathbb{P}(\mathcal{A}^c)+\sup_{A\in\mathcal{A}}\mathbb{P}_A\left(\opnorm{(A\otimes \mathds{1}_d\mathds{1}_d^T)\circ W} > C_2\sqrt{dnp}\right) \\
&\leq& 2n^{-11},
\end{eqnarray*}
which implies the desired result.
\end{proof}

\begin{lemma}\label{lem:talagrand}
Assume $\frac{np}{\log n}>c$ for some sufficiently large constant $c>0$. Consider independent random matrices $X_{ij}\sim\mathcal{MN}(0,I_d,I_d)$ for $1\leq i<j\leq n$. Write $X_{ji}=X_{ij}$ for $1\leq i<j\leq n$. Then, we have
$$\sum_{i=1}^n\fnorm{\sum_{j\in[n]\backslash\{i\}}A_{ij}\left(X_{ij}-X_{ji}\right)}^2\leq 2d(d-1)n(n-1)p+C\left(d^2\sqrt{n^2p\log n}+d\sqrt{n^3p^2\log n}\right),$$
and
$$\sum_{i=1}^n\fnorm{\sum_{j\in[n]\backslash\{i\}}A_{ij}X_{ij}}^2\leq d^2n(n-1)p+C\left(d^2\sqrt{n^2p\log n}+d\sqrt{n^3p^2\log n}\right),$$
with probability at least $1-n^{-10}$ for some constant $C>0$.
\end{lemma}
\begin{proof}
We use $\mathbb{P}_A$ for the conditional probability $\mathbb{P}(\cdot|A)$. Define the event
$$\mathcal{A}=\left\{\max_{i\in[n]}\sum_{j\in[n]\backslash\{i\}}A_{ij}\leq 2np, \sum_{i=1}^n\sum_{j\in[n]\backslash\{i\}}A_{ij}\leq n(n-1)p+5\sqrt{n^2p\log n}\right\}.$$
Under the assumption $\frac{np}{\log n}>c$, we have $\mathbb{P}(\mathcal{A}^c)\leq n^{-11}$ by Bernstein's inequality and a union bound argument.
Define
$$g(X)=\sqrt{\sum_{i=1}^n\fnorm{\sum_{j\in[n]\backslash\{i\}}A_{ij}\left(X_{ij}-X_{ji}\right)}^2}.$$
Then, for any $A\in\mathcal{A}$ and any $X'$ and $X''$, we have
\begin{eqnarray}
\label{eq:lip1} |g(X')-g(X'')| &\leq& \sqrt{\sum_{i=1}^n\fnorm{\sum_{j\in[n]\backslash\{i\}}A_{ij}\left(X_{ij}'-X_{ij}''-X_{ji}'+X_{ji}''\right)}^2} \\
\label{eq:lip2} &\leq& \sqrt{\sum_{i=1}^n\left(\sum_{j\in[n]\backslash\{i\}}A_{ij}\right)\left(\sum_{j\in[n]\backslash\{i\}}\fnorm{X_{ij}'-X_{ij}''-X_{ji}'+X_{ji}''}^2\right)} \\
\nonumber &\leq& \sqrt{2np}\sqrt{\sum_{i=1}^n\sum_{j\in[n]\backslash\{i\}}\fnorm{X_{ij}'-X_{ij}''-X_{ji}'+X_{ji}''}^2} \\
\label{eq:lip3} &\leq& 2\sqrt{2np}\sqrt{\sum_{i=1}^n\sum_{j\in[n]\backslash\{i\}}\fnorm{X_{ij}'-X_{ij}''}^2} \\
\nonumber &=& 4\sqrt{np}\sqrt{\sum_{i=1}^n\sum_{j\in[n]:j>i}\fnorm{X_{ij}'-X_{ij}''}^2}.
\end{eqnarray}
The bounds (\ref{eq:lip1}) and (\ref{eq:lip3}) are due to triangle inequality. The inequality (\ref{eq:lip2}) can be viewed as a generalization of Cauchy-Schwarz, since
\begin{eqnarray*}
&& \fnorm{\sum_{j\in[n]\backslash\{i\}}A_{ij}\left(X_{ij}'-X_{ij}''-X_{ji}'+X_{ji}''\right)}^2 \\
&=& \sup_{K\in\mathbb{R}^{d\times d}:\fnorm{K}=1}\left|\iprod{K}{\sum_{j\in[n]\backslash\{i\}}A_{ij}\left(X_{ij}'-X_{ij}''-X_{ji}'+X_{ji}''\right)}\right|^2 \\
&=& \sup_{K\in\mathbb{R}^{d\times d}:\fnorm{K}=1}\left|\sum_{j\in[n]\backslash\{i\}}A_{ij}\iprod{K}{X_{ij}'-X_{ij}''-X_{ji}'+X_{ji}''}\right|^2 \\
&\leq& \sup_{K\in\mathbb{R}^{d\times d}:\fnorm{K}=1}\left(\sum_{j\in[n]\backslash\{i\}}A_{ij}\right)\left(\sum_{j\in[n]\backslash\{i\}}\iprod{K}{X_{ij}'-X_{ij}''-X_{ji}'+X_{ji}''}^2\right) \\
&\leq& \left(\sum_{j\in[n]\backslash\{i\}}A_{ij}\right)\left(\sum_{j\in[n]\backslash\{i\}}\fnorm{X_{ij}'-X_{ij}''-X_{ji}'+X_{ji}''}^2\right).
\end{eqnarray*}
To summarize, we have shown that $g(X)$ is Lipschitz with respect to $\{X_{ij}\}_{1\leq i<j\leq n}$, and the Lipschitz constant is bounded by $4\sqrt{np}$. By a standard Gaussian concentration inequality for Lipschitz functions \citep{talagrand1995concentration}, we have
\begin{eqnarray*}
&& \mathbb{P}\left(|g(X)-\mathbb{E}(g(X)|A)|>t\sqrt{np}\right) \\
&\leq& \mathbb{P}(\mathcal{A}^c) + \sup_{A\in\mathcal{A}}\mathbb{P}_A\left(|g(X)-\mathbb{E}(g(X)|A)|>t\sqrt{np}\right) \\
&\leq& n^{-11} + 2\exp(-C_1t^2),
\end{eqnarray*}
for some constant $C_1>0$. Therefore, by choosing $t=C_2\sqrt{\log n}$ for some constant $C_2>0$, we have
$$g(X)\leq \mathbb{E}(g(X)|A)+C_2\sqrt{np\log n},$$
with probability at least $1-2n^{-11}$. In addition, we have
\begin{eqnarray*}
\mathbb{E}(g(X)|A) &\leq& \sqrt{\sum_{i=1}^n\mathbb{E}\left(\fnorm{\sum_{j\in[n]\backslash\{i\}}A_{ij}\left(X_{ij}-X_{ji}\right)}^2\Bigg|A\right)} \\
&=& \sqrt{2d(d-1)\sum_{i=1}^n\sum_{j\in[n]\backslash\{i\}}A_{ij}} \\
&\leq& \sqrt{2d(d-1)\left(n(n-1)p+5\sqrt{n^2p\log n}\right)},
\end{eqnarray*}
for any $A\in\mathcal{A}$. Thus,
$$g(X)\leq \sqrt{2d(d-1)\left(n(n-1)p+5\sqrt{n^2p\log n}\right)}+C_2\sqrt{np\log n},$$
with probability at least $1-3n^{-11}$, and the first desired bound is implied by squaring both sides of the above inequality. The second bound can be proved by a similar argument, and we omit the details.
\end{proof}

\begin{lemma}\label{lem:gao}
Consider independent $X_j\sim\mathcal{MN}(0,I_d,I_d)$ and $E_j\sim\text{Bernoulli}(p)$. Then,
$$\mathbb{P}\left(\opnorm{\sum_{j=1}^nE_jX_j}/p>t\right)\leq 81^d\exp\left(-\min\left(\frac{pt^2}{144n},\frac{pt}{6}\right)\right),$$
for any $t>0$.
\end{lemma}
\begin{proof}
Apply a standard discretization trick \citep{vershynin2010introduction}, and there exists a subset $\mathcal{U}\subset\{u\in\mathbb{R}^d:\|u\|=1\}$ with cardinality bound $|\mathcal{U}|\leq 9^d$ that satisfies
$$\opnorm{\sum_{j=1}^nE_jX_j}\leq 3\max_{u,v\in\mathcal{U}}\sum_{j=1}^nE_ju^TX_jv.$$
Note that for any $u,v\in\mathcal{U}$, we have $u^TX_jv\sim\mathcal{N}(0,1)$. By Lemma 13 of \cite{gao2016optimal}, we have
$$\mathbb{P}\left(3\sum_{j=1}^nE_ju^TX_jv/p>t\right)\leq \exp\left(-\min\left(\frac{pt^2}{144n},\frac{pt}{6}\right)\right).$$
Hence,
\begin{eqnarray*}
\mathbb{P}\left(\opnorm{\sum_{j=1}^nE_jX_j}/p>t\right) &\leq& \mathbb{P}\left(3\max_{u,v\in\mathcal{U}}\sum_{j=1}^nE_ju^TX_jv/p>t\right) \\
&\leq& \sum_{u,v\in\mathcal{U}}\mathbb{P}\left(3\sum_{j=1}^nE_ju^TX_jv/p>t\right) \\
&\leq& 81^d\exp\left(-\min\left(\frac{pt^2}{144n},\frac{pt}{6}\right)\right),
\end{eqnarray*}
which is the desired bound.
\end{proof}

\begin{lemma}[Corollary 2.14 of \cite{ipsen2008perturbation}] \label{lem:ipsen}
Consider $X,\wt{X}\in\mathbb{R}^{d\times d}$ with $X$ being full rank. Then,
$$\frac{|\det(X)-\det(\wt{X})|}{|\det(X)|}\leq \left(\opnorm{X^{-1}}\opnorm{X-\wt{X}}+1\right)^d-1.$$
\end{lemma}

\begin{lemma}\label{lem:loss-re}
For any $Z,Z^*\in\o(d)^n$, we have
\begin{equation}
\frac{1}{n^2}\sum_{i=1}^n\sum_{j=1}^n\fnorm{Z_iZ_j^{\T}-Z_i^*Z_j^{*\T}}^2\leq 2\ell(Z,Z^*). \label{eq:loss1-re}
\end{equation}
For any $Z,Z^*\in\so(d)^n$, we have
\begin{equation}
\frac{1}{n^2}\sum_{i=1}^n\sum_{j=1}^n\fnorm{Z_iZ_j^{\T}-Z_i^*Z_j^{*\T}}^2\leq 2\bar{\ell}(Z,Z^*). \label{eq:loss2-re}
\end{equation}
For any $Z,Z^*\in\mathbb{R}^{nd\times d}$ such that $Z/\sqrt{n},Z^*/\sqrt{n}\in\o(nd,d)$, we have
\begin{equation}
\min_{B\in\o(d)}\frac{1}{n}\sum_{i=1}^n\fnorm{Z_i-Z_i^*B}^2 \leq \frac{1}{n^2}\sum_{i=1}^n\sum_{j=1}^n\fnorm{Z_iZ_j^{\T}-Z_i^*Z_j^{*\T}}^2. \label{eq:loss3-re}
\end{equation}
In all inequalities above, we have $Z^{\T}=(Z_1^{\T},\cdots,Z_n^{\T})$ and $Z^{*\T}=(Z_1^{*\T},\cdots,Z_n^{*\T})$, where $Z_i$ and $Z_i^*$ are the $i$th block sub-matrices of size $d\times d$ of $Z$ and $Z^*$, respectively.
\end{lemma}
\begin{proof}
Consider any $Z,Z^*\in\o(d)^n$.
By direct expansion, we can write
$$
\ell(Z,Z^*)=2\left(d-\max_{B\in\o(d)}\Tr\left(\frac{1}{n}\sum_{i=1}^nZ_i^{\T}Z_i^{*}B\right)\right),
$$
and
\begin{equation}
\frac{1}{n^2}\sum_{i=1}^n\sum_{j=1}^n\fnorm{Z_iZ_j^{\T}-Z_i^*Z_j^{*\T}}^2 = 2\left(d-\fnorm{\frac{1}{n}\sum_{i=1}^nZ_i^{\T}Z_i^{*}}^2\right). \label{eq:dive}
\end{equation}
Since
$$\max_{B\in\o(d)}\Tr\left(\frac{1}{n}\sum_{i=1}^nZ_i^{\T}Z_i^{*}B\right)\leq \max_{B\in\mathbb{R}^{d\times d}:\fnorm{B}^2=d}\Tr\left(\frac{1}{n}\sum_{i=1}^nZ_i^{\T}Z_i^{*}B\right)\leq \sqrt{d}\fnorm{\frac{1}{n}\sum_{i=1}^nZ_i^{\T}Z_i^{*}},$$
we have
\begin{eqnarray*}
\frac{1}{n^2}\sum_{i=1}^n\sum_{j=1}^n\fnorm{Z_iZ_j^{\T}-Z_i^*Z_j^{*\T}}^2 &=& 2\left(\sqrt{d}+\fnorm{\frac{1}{n}\sum_{i=1}^nZ_i^{\T}Z_i^{*}}\right)\left(\sqrt{d}-\fnorm{\frac{1}{n}\sum_{i=1}^nZ_i^{\T}Z_i^{*}}\right) \\
&\leq& 4\sqrt{d}\left(\sqrt{d}-\fnorm{\frac{1}{n}\sum_{i=1}^nZ_i^{\T}Z_i^{*}}\right) \\
&\leq& 2\ell(Z,Z^*).
\end{eqnarray*}
The proof of (\ref{eq:loss1-re}) is complete. The inequality (\ref{eq:loss2-re}) can be proved with the same argument.

Finally, we prove (\ref{eq:loss3-re}). For any $Z,Z^*\in\mathbb{R}^{nd\times d}$ such that $Z/\sqrt{n},Z^*/\sqrt{n}\in\o(nd,d)$, we have $\min_{B\in\o(d)}\frac{1}{n}\sum_{i=1}^n\fnorm{Z_i-Z_i^*B}^2=2\left(d-\max_{B\in\o(d)}\Tr\left(\frac{1}{n}\sum_{i=1}^nZ_i^{\T}Z_i^{*}B\right)\right)$ and the identity (\ref{eq:dive}) continues to hold. Suppose the matrix $\frac{1}{n}\sum_{i=1}^nZ_i^{\T}Z_i^{*}$ admits an SVD $\frac{1}{n}\sum_{i=1}^nZ_i^{\T}Z_i^{*}=UDV^{\T}$. Then,
$$\max_{B\in\o(d)}\Tr\left(\frac{1}{n}\sum_{i=1}^nZ_i^{\T}Z_i^{*}B\right)\geq\Tr(UDV^{\T}VU^{\T})=\Tr(D)\geq \frac{\fnorm{\frac{1}{n}\sum_{i=1}^nZ_i^{\T}Z_i^{*}}^2}{\opnorm{\frac{1}{n}\sum_{i=1}^nZ_i^{\T}Z_i^{*}}}.$$
This implies
\begin{eqnarray*}
\min_{B\in\o(d)}\frac{1}{n}\sum_{i=1}^n\fnorm{Z_i-Z_i^*B}^2 &\leq& 2\left(d-\frac{\fnorm{\frac{1}{n}\sum_{i=1}^nZ_i^{\T}Z_i^{*}}^2}{\opnorm{\frac{1}{n}\sum_{i=1}^nZ_i^{\T}Z_i^{*}}}\right) \\
&\leq& 2\left(d-\frac{\fnorm{\frac{1}{n}\sum_{i=1}^nZ_i^{\T}Z_i^{*}}^2}{\opnorm{Z/\sqrt{n}}\opnorm{Z^*/\sqrt{n}}}\right) \\
&\leq& 2\left(d-\fnorm{\frac{1}{n}\sum_{i=1}^nZ_i^{\T}Z_i^{*}}^2\right) \\
&=& \frac{1}{n^2}\sum_{i=1}^n\sum_{j=1}^n\fnorm{Z_iZ_j^{\T}-Z_i^*Z_j^{*\T}}^2.
\end{eqnarray*}
The proof is complete.
\end{proof}

\subsection{Proof of Lemma \ref{lem:main}}

For any $Z\in\o(d)^n$ such that $\ell(Z,Z^*)\leq \gamma$, we define
\begin{equation}
\wt{Z}_i=\frac{\sum_{j\in[n]\backslash\{i\}}A_{ij}Y_{ij}Z_j}{\sum_{j\in[n]\backslash\{i\}}A_{ij}}, \label{eq:def-z-tilde}
\end{equation}
for each $i\in[n]$. Then, write $\wh{Z}=f(Z)$. It is clear that $\wh{Z}_i=\P(\wt{Z}_i)$ for $i$ such that $\wt{Z}_i$ is full rank. The condition $\ell(Z,Z^*)\leq \gamma$ implies that there exists some $B\in \o(d)$ such that $\ell(Z,Z^*)=\frac{1}{n}\sum_{i=1}^n\fnorm{Z_i-Z_i^*B}^2\leq \gamma$. With $Y_{ij}=Z_i^*Z_j^{*\T}+\sigma W_{ij}$, we obtain the following expansion,
\begin{eqnarray*}
Z_i^{*\T}\wt{Z}_iB^{\T} &=& I_d + \frac{\sum_{j\in[n]\backslash\{i\}}A_{ij}Z_j^{*\T}(Z_j-Z_j^*B)B^{\T}}{\sum_{j\in[n]\backslash\{i\}}A_{ij}} \\
&& + \frac{\sigma\sum_{j\in[n]\backslash\{i\}}A_{ij}Z_i^{*\T}W_{ij}(Z_j-Z_j^*B)B^{\T}}{\sum_{j\in[n]\backslash\{i\}}A_{ij}} + \frac{\sigma\sum_{j\in[n]\backslash\{i\}}A_{ij}Z_i^{*\T}W_{ij}Z_j^*}{\sum_{j\in[n]\backslash\{i\}}A_{ij}} \\
&=& I_d + \frac{1}{n-1}\sum_{j=1}^nZ_j^{*\T}(Z_j-Z_j^*B)B^{\T} - \frac{1}{n-1}Z_i^{*\T}(Z_i-Z_i^*B)B^{\T} \\
&& + \frac{\sum_{j\in[n]\backslash\{i\}}A_{ij}Z_j^{*\T}(Z_j-Z_j^*B)B^{\T}}{\sum_{j\in[n]\backslash\{i\}}A_{ij}} - \frac{1}{n-1}\sum_{j\in[n]\backslash\{i\}}Z_j^{*\T}(Z_j-Z_j^*B)B^{\T} \\
&& + \frac{\sigma\sum_{j\in[n]\backslash\{i\}}A_{ij}Z_i^{*\T}W_{ij}(Z_j-Z_j^*B)B^{\T}}{\sum_{j\in[n]\backslash\{i\}}A_{ij}} + \frac{\sigma\sum_{j\in[n]\backslash\{i\}}A_{ij}Z_i^{*\T}W_{ij}Z_j^*}{\sum_{j\in[n]\backslash\{i\}}A_{ij}}.
\end{eqnarray*}
Next, we define
\begin{equation}
\wt{Q}=I_d + \frac{1}{n-1}\sum_{j=1}^nZ_j^{*\T}(Z_j-Z_j^*B)B^{\T}. \label{eq:def-Q-tilde}
\end{equation}
It is clear that
\begin{equation}
\fnorm{\wt{Q}-I_d} \leq \frac{1}{n-1}\sum_{j=1}^n\fnorm{Z_j-Z_j^*B} \leq \frac{n}{n-1}\sqrt{\frac{1}{n}\sum_{j=1}^n\fnorm{Z_j-Z_j^*B}^2}\leq 2\sqrt{\gamma}, \label{eq:Q-I}
\end{equation}
and therefore $\wt{Q}$ is full rank and we can define $Q=\P(\wt{Q})$. With these definitions and the above expansion of $Z_i^{*\T}\wt{Z}_iB^{\T}$, we can write
\begin{equation}
Z_i^{*\T}\wt{Z}_iB^{\T}Q^{\T}=\wt{Q}Q^{\T}+S_i, \label{eq:pf-expan}
\end{equation}
where $S_i=- \frac{1}{n-1}Z_i^{*\T}(Z_i-Z_i^*B)B^{\T}Q^{\T}+F_i+G_i+H_i$, with
\begin{eqnarray*}
F_i &=& \frac{\sum_{j\in[n]\backslash\{i\}}A_{ij}Z_j^{*\T}(Z_j-Z_j^*B)B^{\T}Q^{\T}}{\sum_{j\in[n]\backslash\{i\}}A_{ij}} - \frac{1}{n-1}\sum_{j\in[n]\backslash\{i\}}Z_j^{*\T}(Z_j-Z_j^*B)B^{\T}Q^{\T}, \\
G_i &=& \frac{\sigma\sum_{j\in[n]\backslash\{i\}}A_{ij}Z_i^{*\T}W_{ij}(Z_j-Z_j^*B)B^{\T}Q^{\T}}{\sum_{j\in[n]\backslash\{i\}}A_{ij}}, \\
H_i &=& \frac{\sigma\sum_{j\in[n]\backslash\{i\}}A_{ij}Z_i^{*\T}W_{ij}Z_j^*Q^{\T}}{\sum_{j\in[n]\backslash\{i\}}A_{ij}}.
\end{eqnarray*}
For the first term of (\ref{eq:pf-expan}), it is clear that $\wt{Q}Q^{\T}$ is a symmetric matrix by the definition of $\mathcal{P}(\cdot)$ given in (\ref{eq:polar-svd}).
We also have
\begin{equation}
\fnorm{\frac{1}{n-1}Z_i^{*\T}(Z_i-Z_i^*B)B^{\T}Q^{\T}} = \frac{1}{n-1}\fnorm{Z_i-Z_i^*B} \leq \frac{\sqrt{n}}{n-1}\sqrt{\frac{1}{n}\sum_{i=1}^n\fnorm{Z_i-Z_i^*B}^2}\leq \sqrt{\gamma}. \label{eq:small-t}
\end{equation}
As long as $\opnorm{F_i}\vee\opnorm{G_i}\vee\opnorm{H_i}\leq\rho$, by (\ref{eq:Q-I}), (\ref{eq:pf-expan}) and (\ref{eq:small-t}), we have
\begin{equation}
s_{\min}(Z_i^{*\T}\wt{Z}_iB^{\T}Q^{\T})\geq s_{\min}(\wt{Q})-\opnorm{S_i}\geq 1-3(\rho+\sqrt{\gamma}). \label{eq:s-min}
\end{equation}
When $3(\rho+\sqrt{\gamma})<1$, $Z_i^{*\T}\wt{Z}_iB^{\T}Q^{\T}$ has full rank, and so does $\wt{Z}_i$, and thus we have $\wh{Z}_i=\P(\wt{Z}_i)$.
Now we can apply a perturbation analysis to (\ref{eq:pf-expan}),
\begin{eqnarray}
\nonumber \fnorm{\wh{Z}_i-Z_i^*QB}^2 &=& \fnorm{\P(\wt{Z}_i)-Z_i^*QB}^2 \\
\label{eq:crit1} &=& \fnorm{\P(Z_i^{*\T}\wt{Z}_iB^{\T}Q^{\T})-I_d}^2 \\
\label{eq:crit2} &=& \fnorm{\P(\wt{Q}Q^{\T}+S_i)-\P\left(\wt{Q}Q^{\T}+\frac{1}{2}S_i+\frac{1}{2}S_i^{\T}\right)}^2 \\
\label{eq:crit3} &\leq& \frac{1}{[1-3(\rho+\sqrt{\gamma})]^2}\fnorm{\frac{1}{2}S_i-\frac{1}{2}S_i^{\T}}^2.
\end{eqnarray}
The equality (\ref{eq:crit1}) is by the properties of $\P(\cdot)$ listed in Section \ref{sec:alg}. We then used the fact that $\wt{Q}Q^{\T}+\frac{1}{2}S_i+\frac{1}{2}S_i^{\T}$ is symmetric and positive definite so that $\P\left(\wt{Q}Q^{\T}+\frac{1}{2}S_i+\frac{1}{2}S_i^{\T}\right)=I_d$, which leads to (\ref{eq:crit2}). The inequality (\ref{eq:crit3}) is by Lemma \ref{lem:li}, with $s_{\min}(\wt{Q}Q^{\T}+S_i)$ lower bounded by (\ref{eq:s-min}) and $s_{\min}\left(\wt{Q}Q^{\T}+\frac{1}{2}S_i+\frac{1}{2}S_i^{\T}\right)$ is lower bounded similarly (which also implies that $\wt{Q}Q^{\T}+\frac{1}{2}S_i+\frac{1}{2}S_i^{\T}$ is positive definite required by (\ref{eq:crit2})). The perturbation analysis has been done under the condition $\opnorm{F_i}\vee\opnorm{G_i}\vee\opnorm{H_i}\leq\rho$. Without this condition, we have
\begin{eqnarray}
\nonumber \fnorm{\wh{Z}_i-Z_i^*QB}^2 &\leq& \frac{1}{[1-3(\rho+\sqrt{\gamma})]^2}\fnorm{\frac{1}{2}S_i-\frac{1}{2}S_i^{\T}}^2\mathbb{I}\{\opnorm{F_i}\vee\opnorm{G_i}\vee\opnorm{H_i}\leq\rho\} \\
\nonumber && + 4d\mathbb{I}\{\opnorm{F_i}\vee\opnorm{G_i}\vee\opnorm{H_i}>\rho\} \\
\nonumber &\leq& \frac{1+\eta}{[1-3(\rho+\sqrt{\gamma})]^2}\fnorm{\frac{1}{2}H_i-\frac{1}{2}H_i^{\T}}^2 + \frac{3(1+\eta^{-1})}{[1-3(\rho+\sqrt{\gamma})]^2}\left(\fnorm{F_i}^2+\fnorm{G_i}^2\right) \\
\nonumber && + \frac{3(1+\eta^{-1})}{[1-3(\rho+\sqrt{\gamma})]^2}\fnorm{\frac{1}{n-1}Z_i^{*\T}(Z_i-Z_i^*B)B^{\T}Q^{\T}}^2 \\
\label{eq:master} && + 4d\mathbb{I}\{\opnorm{F_i}>\rho\} + 4d\mathbb{I}\{\opnorm{G_i}>\rho\} + 4d\mathbb{I}\{\opnorm{H_i}>\rho\},
\end{eqnarray}
where we have used the inequality $\fnorm{X+X'}^2\leq (1+\eta)\fnorm{X}^2+(1+\eta^{-1})\fnorm{X'}^2$ for any matrices $X,X'\in\mathbb{R}^{d\times d}$.
The above bound holds deterministically for any $\eta\in(0,1)$ and any $\rho>0$ that satisfies $3(\rho+\sqrt{\gamma})<1$. The specific values of $\eta$ and $\rho$ will be determined later.

In the next step of the proof, we will need to analyze $F_i$, $G_i$ and $H_i$. This requires a few high probability events.
First, by Lemma \ref{lem:gao}, we have
$$
\sum_{i=1}^n\mathbb{P}\left(\frac{2\sigma}{np}\opnorm{\sum_{j\in[n]\backslash\{i\}}A_{ij}W_{ij}Z_j^*}>\rho\right) \leq 81^dn\exp\left(-\min\left(\frac{\rho^2np}{576\sigma^2},\frac{\rho np}{12\sigma}\right)\right).
$$
By Markov inequality, we have
\begin{equation}
\sum_{i=1}^n\mathbb{I}\left\{\frac{2\sigma}{np}\opnorm{\sum_{j\in[n]\backslash\{i\}}A_{ij}W_{ij}Z_j^*}>\rho\right\} \leq \frac{\sigma^2}{\rho^2p}\exp\left(-\sqrt{\frac{\rho^2 np}{\sigma^2}}\right), \label{eq:hollow0}
\end{equation}
with probability at least
\begin{eqnarray*}
&& 1-81^d\frac{\rho^2np}{\sigma^2}\left[\exp\left(-\frac{\rho^2np}{576\sigma^2}+\sqrt{\frac{\rho^2np}{\sigma^2}}\right)+\exp\left(-\frac{\rho np}{\sigma^2}+\sqrt{\frac{\rho^2 np}{\sigma^2}}\right)\right] \\
&\geq& 1-81^d\frac{2\rho^2 np}{\sigma^2}\exp\left(-3\sqrt{\frac{\rho^2 np}{\sigma^2}}\right) \\
&\geq& 1-81^d\exp\left(-2\sqrt{\frac{\rho^2np}{\sigma^2}}\right) \\
&\geq& 1-\exp\left(-\sqrt{\frac{\rho^2np}{\sigma^2}}\right).
\end{eqnarray*}
The above set of inequalities requires that $\rho$ satisfies $\sqrt{\frac{\rho^2 np}{\sigma^2}}>2304+5d$. Then, by Lemma \ref{lem:ER-graph}, Lemma \ref{lem:bandeira} and Lemma \ref{lem:talagrand}, we know that 
\begin{eqnarray}
\label{eq:hollow1} \min_{i\in[n]}\sum_{j\in[n]\backslash\{i\}}A_{ij} &\geq& (n-1)p - C\sqrt{np\log n}, \\
\label{eq:hollow2}\max_{i\in[n]}\sum_{j\in[n]\backslash\{i\}}A_{ij} &\leq& (n-1)p + C\sqrt{np\log n}, \\
\label{eq:hollow3}\opnorm{A-\mathbb{E}A} &\leq& C\sqrt{np}, \\
\label{eq:hollow4}\opnorm{(A\otimes \mathds{1}_d\mathds{1}_d^T)\circ W} &\leq& C\sqrt{dnp}, \\
\label{eq:hollow6}\sum_{i=1}^n\fnorm{\sum_{j\in[n]\backslash\{i\}}A_{ij}\left(Z_i^{*\T}W_{ij}Z_j^*-Z_j^{*\T}W_{ji}Z_i^*\right)}^2 &\leq& 2d(d-1)n^2p\left(1+C\sqrt{\frac{\log n}{n}}\right), \\
\label{eq:hollow7}\sum_{i=1}^n\fnorm{\sum_{j\in[n]\backslash\{i\}}A_{ij}W_{ij}Z_j^*}^2 &\leq& d^2n^2p\left(1+C\sqrt{\frac{\log n}{n}}\right),
\end{eqnarray}
all hold with probability at least $1-n^{-9}$ for some constant $C>0$. We conclude that the events (\ref{eq:hollow0}), (\ref{eq:hollow1}), (\ref{eq:hollow2}), (\ref{eq:hollow3}), (\ref{eq:hollow4}), (\ref{eq:hollow6}) and (\ref{eq:hollow7}) hold simultaneously with probability at least $1-n^{-9}-\exp\left(-\sqrt{\frac{\rho^2 np}{\sigma^2}}\right)$. These events will be assumed from now on.

\textit{Analysis of $F_i$.} By triangle inequality, (\ref{eq:hollow1}) and (\ref{eq:hollow2}), we have
\begin{eqnarray*}
\fnorm{F_i} &\leq& \frac{\fnorm{\sum_{j\in[n]\backslash\{i\}}(A_{ij}-p)Z_j^{*\T}(Z_j-Z_j^*B)}}{\sum_{j\in[n]\backslash\{i\}}A_{ij}} \\
&& + \left|\frac{p}{\sum_{j\in[n]\backslash\{i\}}A_{ij}}-\frac{1}{n-1}\right|\fnorm{\sum_{j\in[n]\backslash\{i\}}Z_j^{*\T}(Z_j-Z_j^*B)} \\
&\leq& \frac{2}{np}\fnorm{\sum_{j\in[n]\backslash\{i\}}(A_{ij}-p)Z_j^{*\T}(Z_j-Z_j^*B)} \\
&& + \frac{2\left|\sum_{j\in[n]\backslash\{i\}}(A_{ij}-p)\right|}{n^2p}\sum_{j\in[n]\backslash\{i\}}\fnorm{Z_j-Z_j^*B} \\
&\leq& \frac{2}{np}\fnorm{\sum_{j\in[n]\backslash\{i\}}(A_{ij}-p)Z_j^{*\T}(Z_j-Z_j^*B)}  + \frac{C_1\sqrt{p\log n}}{np}\sqrt{\sum_{i=1}^n\fnorm{Z_i-Z_i^*B}^2}.
\end{eqnarray*}
Using (\ref{eq:hollow3}), we have
\begin{eqnarray}
\nonumber \frac{1}{n}\sum_{i=1}^n\fnorm{F_i}^2 &\leq& \frac{8}{n^3p^2}\sum_{i=1}^n\fnorm{\sum_{j\in[n]\backslash\{i\}}(A_{ij}-p)Z_j^{*\T}(Z_j-Z_j^*B)} ^2 + 2C_1^2\frac{\log n}{np}\ell(Z,Z^*) \\
\nonumber &\leq& \frac{8}{n^3p^2}\opnorm{A-\mathbb{E}A}^2\sum_{i=1}^n\fnorm{Z_i-Z_i^*B}^2 + 2C_1^2\frac{\log n}{np}\ell(Z,Z^*) \\
\label{eq:f1} &\leq& C_2\frac{\log n}{np}\ell(Z,Z^*).
\end{eqnarray}
This bound also implies
\begin{equation}
\frac{1}{n}\sum_{i=1}^n\mathbb{I}\{\opnorm{F_i}>\rho\}\leq \frac{1}{n\rho^2}\sum_{i=1}^n\fnorm{F_i}^2\leq C_2\frac{\log n}{\rho^2np}\ell(Z,Z^*). \label{eq:f2}
\end{equation}

\textit{Analysis of $G_i$.} By (\ref{eq:hollow1}) and (\ref{eq:hollow4}), we have
\begin{eqnarray}
\nonumber \frac{1}{n}\sum_{i=1}^n\fnorm{G_i}^2 &\leq& \frac{2\sigma^2}{n^3p^2}\sum_{i=1}^n\fnorm{\sum_{j\in[n]\backslash\{i\}}A_{ij}W_{ij}(Z_j-Z_j^*B)}^2 \\
\nonumber &\leq& \frac{2\sigma^2}{n^3p^2}\opnorm{(A\otimes \mathds{1}_d\mathds{1}_d^T)\circ W}^2\sum_{i=1}^n\fnorm{Z_i-Z_i^*B}^2 \\
\label{eq:g1} &\leq& C_3\frac{\sigma^2d}{np}\ell(Z,Z^*),
\end{eqnarray}
and thus
\begin{equation}
\frac{1}{n}\sum_{i=1}^n\mathbb{I}\{\opnorm{G_i}>\rho\}\leq \frac{1}{n\rho^2}\sum_{i=1}^n\fnorm{G_i}^2\leq C_3\frac{\sigma^2d}{\rho^2np}\ell(Z,Z^*). \label{eq:g2}
\end{equation}

\textit{Analysis of $H_i$.} First, by (\ref{eq:hollow0}) and (\ref{eq:hollow1}), we have
\begin{eqnarray}
\nonumber \frac{1}{n}\sum_{i=1}^n\mathbb{I}\{\opnorm{H_i}>\rho\} &\leq& \frac{1}{n}\sum_{i=1}^n\mathbb{I}\left\{\frac{2\sigma}{np}\opnorm{\sum_{j\in[n]\backslash\{i\}}A_{ij}W_{ij}Z_j^*}>\rho\right\} \\
\nonumber &\leq& \frac{\sigma^2}{\rho^2np}\exp\left(-\sqrt{\frac{\rho^2 np}{\sigma^2}}\right) \\
\label{eq:h1}&\leq& \exp\left(-\frac{1}{2}\sqrt{\frac{\rho^2 np}{\sigma^2}}\right).
\end{eqnarray}
Next, to bound $\fnorm{\frac{1}{2}H_i-\frac{1}{2}H_i^{\T}}^2$, we introduce the notation $E_i=\frac{\sigma\sum_{j\in[n]\backslash\{i\}}A_{ij}Z_i^{*\T}W_{ij}Z_j^*}{\sum_{j\in[n]\backslash\{i\}}A_{ij}}$ and thus we can write $H_i=E_iQ^{\T}$. We then have
\begin{eqnarray*}
\frac{1}{n}\sum_{i=1}^n\fnorm{\frac{1}{2}H_i-\frac{1}{2}H_i^{\T}}^2 &=& \frac{1}{n}\sum_{i=1}^n\fnorm{\frac{1}{2}E_i-\frac{1}{2}E_i^{\T}+\frac{1}{2}E_i(Q^{\T}-I_d)-\frac{1}{2}(Q-I_d)E_i^{\T}}^2 \\
&\leq& (1+\eta)\frac{1}{n}\sum_{i=1}^n\fnorm{\frac{1}{2}E_i-\frac{1}{2}E_i^{\T}}^2 + (1+\eta^{-1})\frac{1}{n}\sum_{i=1}^n\fnorm{E_i(Q^{\T}-I_d)}^2 \\
&\leq& (1+\eta)\frac{1}{n}\sum_{i=1}^n\fnorm{\frac{1}{2}E_i-\frac{1}{2}E_i^{\T}}^2 + (1+\eta^{-1})\fnorm{Q-I_d}^2\frac{1}{n}\sum_{i=1}^n\fnorm{E_i}^2.
\end{eqnarray*}
By (\ref{eq:hollow1}) and (\ref{eq:hollow6}), we have
\begin{eqnarray*}
&& \frac{1}{n}\sum_{i=1}^n\fnorm{\frac{1}{2}E_i-\frac{1}{2}E_i^{\T}}^2 \\
&\leq& \frac{\sigma^2}{4\left(np-2C\sqrt{np\log n}\right)^2}\frac{1}{n}\sum_{i=1}^n\fnorm{\sum_{j\in[n]\backslash\{i\}}A_{ij}\left(Z_i^{*\T}W_{ij}Z_j^*-Z_j^{*\T}W_{ji}Z_i^*\right)}^2 \\
&\leq& \left(1+C_5\sqrt{\frac{\log n}{np}}\right)\frac{\sigma^2d(d-1)}{2np}.
\end{eqnarray*}
By Lemma \ref{lem:li} and (\ref{eq:Q-I}), we have
$$\fnorm{Q-I_d}\leq 2\fnorm{\wt{Q}-I_d}\leq 4\sqrt{\gamma}.$$
By (\ref{eq:hollow1}) and (\ref{eq:hollow7}), we have
\begin{eqnarray*}
\frac{1}{n}\sum_{i=1}^n\fnorm{E_i}^2 &\leq& \frac{\sigma^2}{\left(np-2C\sqrt{np\log n}\right)^2}\frac{1}{n}\sum_{i=1}^n\fnorm{\sum_{j\in[n]\backslash\{i\}}A_{ij}Z_i^{*\T}W_{ij}Z_j^*}^2 \\
&\leq& \frac{2\sigma^2d^2}{np}.
\end{eqnarray*}
We combine the three bounds above and obtain
\begin{equation}
\frac{1}{n}\sum_{i=1}^n\fnorm{\frac{1}{2}H_i-\frac{1}{2}H_i^{\T}}^2 \leq (1+\eta)\left(1+C_5\sqrt{\frac{\log n}{np}}\right)\frac{\sigma^2d(d-1)}{2np} + (1+\eta^{-1})\frac{16\gamma\sigma^2d^2}{np}. \label{eq:h2}
\end{equation}

Finally, we also have the bound
\begin{equation}
\frac{1}{n}\sum_{i=1}^n\fnorm{\frac{1}{n-1}Z_i^{*\T}(Z_i-Z_i^*B)B^{\T}Q^{\T}}^2 \leq \frac{1}{(n-1)^2}\ell(Z,Z^*). \label{eq:super-small}
\end{equation}
Now we can plug the bounds (\ref{eq:f1}), (\ref{eq:f2}), (\ref{eq:g1}), (\ref{eq:g2}), (\ref{eq:h1}), (\ref{eq:h2}) and (\ref{eq:super-small}) into (\ref{eq:master}), and we have
\begin{eqnarray*}
\ell(\wh{Z},Z^*) &\leq& \frac{1}{n}\sum_{i=1}^n\fnorm{\wh{Z}_i-Z_i^*QB}^2 \\
&\leq& \left(1+C_6\left(\rho +\sqrt{\gamma}+\eta+\eta^{-1}\gamma+\sqrt{\frac{\log n}{np}}\right)\right)\frac{\sigma^2d(d-1)}{2np} \\
&& + 4d\exp\left(-\frac{1}{2}\sqrt{\frac{\rho^2np}{\sigma^2}}\right) + C_6\left(\eta^{-1}+d\rho^{-2}\right)\frac{\log n+\sigma^2d}{np}\ell(Z,Z^*).
\end{eqnarray*}
So far we have required $\eta\in(0,1)$, $\rho>0$, $3(\rho+\sqrt{\gamma})<1$ and $\sqrt{\frac{\rho^2 np}{\sigma^2}}>2304+5d$. 
Set $$\eta=\sqrt{\gamma+\frac{\log n+\sigma^2d}{np}}\quad\text{and}\quad \rho^2=\sqrt{\frac{d\log n+\sigma^2d^2}{np}}.$$
Under the conditions that $\gamma<16^{-1}$, $\frac{d\log n + \sigma^2d^2}{np}$ upper bounded by a sufficiently small constant, all the  requirements are satisfied. With this choice, we have
$$\ell(Z,Z^*)\leq\left(1+C_7\left(\gamma^2+\frac{\log n+ \sigma^2d}{np}\right)^{1/4}\right)\frac{\sigma^2d(d-1)}{2np}+C_7\sqrt{\frac{d\log n+\sigma^2d^2}{np}}\ell(Z,Z^*).$$


Note that this inequality has been derived from the events (\ref{eq:hollow0}), (\ref{eq:hollow1}), (\ref{eq:hollow2}), (\ref{eq:hollow3}), (\ref{eq:hollow4}), (\ref{eq:hollow6}) and (\ref{eq:hollow7}) and $\ell(Z,Z^*)\leq \gamma$. Thus, it holds uniformly over all $Z\in\o(d)^n$ such that $\ell(Z,Z^*)\leq \gamma$ with probability at least $1-n^{-9}-\exp\left(-\br{\frac{ np}{\sigma^2}}^{1/4}\right)$. The proof is complete.

\subsection{Proof of Theorem \ref{thm:upper-sod}}\label{sec:pf-sod}

We first characterize the evolution of $\bar{\ell}(Z^{(t)},Z^*)$ through the map
$$Z^{(t)}=\bar{f}(Z^{(t-1)}),$$
where $\bar{f}:\so(d)^n\rightarrow\so(d)^n$ is defined by (\ref{eq:sos-iter}).
\begin{lemma}\label{lem:main-2}
For the $\so(d)$ Synchronization (\ref{eq:sod-setting}),
assume $\frac{np}{\sigma^2}\geq c_1$ and $\frac{np}{\log n}\geq c_2$ for some sufficiently large constants $c_1,c_2>0$ and $d\leq C$ for some constant $C>0$. Then, for any $\gamma\in[0,1/(32d^2))$, we have
\begin{eqnarray*}
&& \mathbb{P}\left(\bar{\ell}(\bar{f}(Z),Z^*)\leq \delta_1\bar{\ell}(Z,Z^*)+(1+\delta_2)\frac{\sigma^2d(d-1)}{2np}\text{ for all }Z\in\so(d)^n\text{ such that }\bar{\ell}(Z,Z^*)\leq\gamma\right) \\
&\geq& 1-n^{-9}-\exp\left(-\br{\frac{ np}{\sigma^2}}^{1/4}\right),
\end{eqnarray*}
where $\delta_1=C_1\sqrt{\frac{\log n+\sigma^2}{np}}$ and $\delta_2=C_2\left(\gamma^2+\frac{\log n+\sigma^2}{np}\right)^{1/4}$ for some constants $C_1,C_2>0$ that only depend on $C$.
\end{lemma}
\begin{proof}
The proof follows the same argument as that of Lemma \ref{lem:main}, and we only point out the difference. For any $Z\in\so(d)^n$ such that $\bar{\ell}(Z,Z^*)\leq\gamma$, define $\wt{Z}$ according to (\ref{eq:def-z-tilde}). The condition $\bar{\ell}(Z,Z^*)\leq\gamma$ implies that there exists some $B\in\so(d)$ such that $\bar{\ell}(Z,Z^*)=\frac{1}{n}\sum_{i=1}^n\fnorm{Z_i-Z_i^*B}^2\leq \gamma$. Recall the definition of $\wt{Q}$ in (\ref{eq:def-Q-tilde}) and the bound $\fnorm{\wh{Q}-I_d}\leq 2\sqrt{\gamma}$ that it satisfies according to (\ref{eq:Q-I}). By Lemma \ref{lem:ipsen}, we have
$$|\det(\wt{Q})-1|\leq \left(\fnorm{\wt{Q}-I_d}+1\right)^d-1\leq (2\sqrt{\gamma}+1)^d-1<1.$$
The last inequality is by the condition $2\sqrt{\gamma}\leq e^{d^{-1}\log 2}-1$. Then, we can conclude that $\det(\wt{Q})>0$. By the definition of $\bar{\P}(\cdot)$, we then have $Q=\P(\wt{Q})=\bar{\P}(\wt{Q})\in\so(d)$. Recall the definitions of $S_1$, $F_i$, $G_i$ and $H_i$ in the proof of Lemma \ref{lem:main}. By (\ref{eq:pf-expan}), we have
$$\opnorm{Z_i^{*\T}\wt{Z}_iB^{\T}-I_d}\leq \opnorm{Z_i^{*\T}\wt{Z}_iB^{\T}Q^{\T}-\wt{Q}Q^{\T}}+\opnorm{\wt{Q}-I_d}\leq 3(\rho+\sqrt{\gamma}),$$
as long as $\opnorm{F_i}\vee\opnorm{G_i}\vee\opnorm{H_i}\leq\rho$. By Lemma \ref{lem:ipsen}, we have
$$
|\det(Z_i^{*\T}\wt{Z}_iB^{\T})-1| \leq \left(\opnorm{Z_i^{*\T}\wt{Z}_iB^{\T}-I_d}+1\right)^d-1\leq (3(\rho+\sqrt{\gamma})+1)^d-1<1.
$$
The last inequality requires $3(\rho+\sqrt{\gamma})<e^{d^{-1}\log 2}-1$. We then have $\det(Z_i^{*\T}\wt{Z}_iB^{\T})>0$, which implies $\det(\wt{Z}_i)>0$. Thus, $\bar{\P}(\wt{Z}_i)=\P(\wt{Z}_i)$. Together with the fact $Q\in\so(d)$ that we have just established, the remaining arguments in the proof of Lemma \ref{lem:main} can be directly applied to obtain the desired conclusion. The only subtle difference is that here we require $3(\rho+\sqrt{\gamma})<e^{d^{-1}\log 2}-1$ compared with $3(\rho+\sqrt{\gamma})<1$ in the proof of Lemma \ref{lem:main}. As a result, a sufficient condition for $\gamma$ is $\gamma < 1/(32d^2)$ compared with $\gamma<1/16$ in Lemma \ref{lem:main}. This detail does not affect the result given the assumption that $d\leq C$.
\end{proof}

Then, we establish the error bound for the initialization procedure.
\begin{lemma}\label{lem:sod-ini}
For the $\so(d)$ Synchronization (\ref{eq:sod-setting}),
assume $\frac{np}{\log n}\geq c$ for some sufficiently large constants $c>0$ and $d\leq C_1$ for some constant $C_1>0$. Then, for $Z^{(0)}$ defined in (\ref{eq:sod-ini}), we have
$$\bar{\ell}(Z^{(0)},Z^*)\leq C\frac{\sigma^2+1}{np},$$
with probability at least $1-n^{-9}$ for some constant $C>0$ only depending on $C_1$.
\end{lemma}
The proof of Lemma \ref{lem:sod-ini} will be given in Section \ref{sec:pf-ini}.

\begin{proof}[Proof of Theorem \ref{thm:upper-sod}]
By Lemma \ref{lem:main-2} and the same argument that leads to (\ref{eq:linear}), we have
$$
\bar{\ell}(Z^{(t)},Z^*) \leq \delta_1^t\bar{\ell}(Z^{(0)},Z^*)+\frac{1+\delta_2}{1-\delta_1}\frac{\sigma^2d(d-1)}{2np},
$$
for all $t\geq 1$, as long as $\bar{\ell}(Z^{(0)},Z^*)$ is sufficiently small. The initial error $\bar{\ell}(Z^{(0)},Z^*)$ is controlled by Lemma \ref{lem:sod-ini}, and thus the desired result follows.
\end{proof}

\subsection{Proofs of Lemma \ref{lem:od-ini} and Lemma \ref{lem:sod-ini}}\label{sec:pf-ini}

We first state a property of the operator $\bar{\P}(\cdot)$.
\begin{lemma}\label{lem:P-bar-prop}
Consider a full-rank $X=(X_1,\cdots,X_d)\in\mathbb{R}^{d\times d}$, where $X_a\in\mathbb{R}^d$ is the $a$th column of $X$. Define $\wt{X}=(X_1,\cdots,X_{d-1},-X_d)$ by changing the sign of the last column of $X$. Then, we have $\bar{\P}(X)=\bar{\P}(\wt{X}).$
\end{lemma}
\begin{proof}
Suppose $X$ admits an SVD $X=UDV^{\T}$. Define $\wt{V}^{\T}$ by changing the sign of the last column of $V^{\T}$. Then, the SVD of $\wt{X}$ is $\wt{X}=UD\wt{V}^{\T}$, and we have
$$\bar{\mathcal{P}}(X)=U\begin{pmatrix}
1 &  & & \\
   & \ddots & & \\
 & & 1 & \\
 & & & \det(UV^{\T})
\end{pmatrix}V^{\T}=U\begin{pmatrix}
1 &  & & \\
   & \ddots & & \\
 & & 1 & \\
 & & & -\det(UV^{\T})
\end{pmatrix}\wt{V}^{\T}=\bar{\P}(\wt{X}),$$
where the last equality is by $-\det(UV^{\T})=\det(U\wt{V}^{\T})$.
\end{proof}

\begin{proof}[Proofs of Lemma \ref{lem:od-ini} and Lemma \ref{lem:sod-ini}]
For any $Z\in\mathbb{R}^{nd\times d}$ such that $Z/\sqrt{n}\in\o(nd,d)$, we have
$$\fnorm{p^{-1}(A\otimes \mathds{1}_d\mathds{1}_d^T)\circ Y-ZZ^{\T}}^2=\fnorm{p^{-1}(A\otimes \mathds{1}_d\mathds{1}_d^T)\circ Y}^2+n^2d-2\Tr\left((p^{-1}(A\otimes \mathds{1}_d\mathds{1}_d^T)\circ Y)ZZ^{\T}\right).$$
Therefore, minimizing $\fnorm{p^{-1}(A\otimes \mathds{1}_d\mathds{1}_d^T)\circ Y-ZZ^{\T}}^2$ is equivalent to maximizing $\Tr\left((p^{-1}(A\otimes \mathds{1}_d\mathds{1}_d^T)\circ Y)ZZ^{\T}\right)$. For $\o(d)$ synchronization, we can thus write $Z^{(0)}$ as
$$Z_i^{(0)}=\begin{cases}
\P(\wh{Z}_i), & \det(\wh{Z}_i)\neq 0, \\
I_d, &\det(Z_i)=0,
\end{cases}$$
where
$$
\wh{Z} = \argmin_{Z\in\mathbb{R}^{nd\times d}: Z/\sqrt{n}\in\o(nd,d)}\fnorm{p^{-1}(A\otimes \mathds{1}_d\mathds{1}_d^T)\circ Y-ZZ^{\T}}^2.
$$
For $\so(d)$ synchronization, $Z^{(0)}$ is
$$Z_i^{(0)}=\begin{cases}
\bar{\P}(\wh{Z}_i), & \det(\wh{Z}_i)\neq 0, \\
I_d, &\det(Z_i)=0,
\end{cases}$$
where $\wh{Z}$ has the same definition.

Our first step is to derive a bound for $\wh{Z}$. For both $\o(d)$ synchronization and $\so(d)$ synchronization, it is clear that $Z^*/\sqrt{n}\in\o(nd,d)$. By the definition of $\wh{Z}$, we have
$$\fnorm{p^{-1}(A\otimes \mathds{1}_d\mathds{1}_d^T)\circ Y-\wh{Z}\wh{Z}^{\T}}^2\leq \fnorm{p^{-1}(A\otimes \mathds{1}_d\mathds{1}_d^T)\circ Y-Z^*Z^{*\T}}^2.$$
After rearrangement, we obtain
$$\fnorm{\wh{Z}\wh{Z}^{\T}-Z^*Z^{*\T}}^2\leq 2\left|\Tr\left((\wh{Z}\wh{Z}^{\T}-Z^*Z^{*\T})(p^{-1}(A\otimes \mathds{1}_d\mathds{1}_d^T)\circ Y-Z^*Z^{*\T})\right)\right|.$$
This implies
$$\fnorm{\wh{Z}\wh{Z}^{\T}-Z^*Z^{*\T}}\leq 2\left|\Tr\left(\frac{\wh{Z}\wh{Z}^{\T}-Z^*Z^{*\T}}{\fnorm{\wh{Z}\wh{Z}^{\T}-Z^*Z^{*\T}}}(p^{-1}(A\otimes \mathds{1}_d\mathds{1}_d^T)\circ Y-Z^*Z^{*\T})\right)\right|.$$
Note that $\frac{\wh{Z}\wh{Z}^{\T}-Z^*Z^{*\T}}{\fnorm{\wh{Z}\wh{Z}^{\T}-Z^*Z^{*\T}}}$ is a matrix of rank at most $2d$, and thus it admits an eigendecomposition $\frac{\wh{Z}\wh{Z}^{\T}-Z^*Z^{*\T}}{\fnorm{\wh{Z}\wh{Z}^{\T}-Z^*Z^{*\T}}}=\sum_{j=1}^{2d}\lambda_juu_j^{\T}$ with the eigenvalues satisfying $\sum_{j=1}^{2d}\lambda_j^2=1$. Then, we have
\begin{eqnarray*}
&& \left|\Tr\left(\frac{\wh{Z}\wh{Z}^{\T}-Z^*Z^{*\T}}{\fnorm{\wh{Z}\wh{Z}^{\T}-Z^*Z^{*\T}}}(p^{-1}(A\otimes \mathds{1}_d\mathds{1}_d^T)\circ Y-Z^*Z^{*\T})\right)\right| \\
&\leq& \sum_{j=1}^{2d}|\lambda_j|\left|u_j^{\T}(p^{-1}(A\otimes \mathds{1}_d\mathds{1}_d^T)\circ Y-Z^*Z^{*\T})u_j\right| \\
&\leq& \opnorm{p^{-1}(A\otimes \mathds{1}_d\mathds{1}_d^T)\circ Y-Z^*Z^{*\T}}\sum_{j=1}^{2d}|\lambda_j| \\
&\leq& \sqrt{2d}\opnorm{p^{-1}(A\otimes \mathds{1}_d\mathds{1}_d^T)\circ Y-Z^*Z^{*\T}}.
\end{eqnarray*}
Hence,
\begin{eqnarray*}
&& \fnorm{\wh{Z}\wh{Z}^{\T}-Z^*Z^{*\T}} \\
&\leq& 2\sqrt{2d}\opnorm{p^{-1}(A\otimes \mathds{1}_d\mathds{1}_d^T)\circ Y-Z^*Z^{*\T}} \\
&\leq& 2\sqrt{2d}\frac{1}{p}\opnorm{((A-\mathbb{E}A)\otimes \mathds{1}_d\mathds{1}_d^T)\circ Z^*Z^{*\T}} + 2\sqrt{2d}\frac{\sigma}{p}\opnorm{(A\otimes \mathds{1}_d\mathds{1}_d^T)\circ W}.
\end{eqnarray*}
For the first term, we have
\begin{eqnarray*}
&& \opnorm{((A-\mathbb{E}A)\otimes \mathds{1}_d\mathds{1}_d^T)\circ Z^*Z^{*\T}} \\
&=& \max_{u^{\T}=(u_1^{\T},\cdots,u_n^{\T}):\sum_{i=1}^n\|u_i\|^2=1}\left|\sum_{i=1}^n\sum_{j=1}^n(A_{ij}-\mathbb{E}A_{ij})u_i^{\T}Z_iZ_j^{\T}u_j\right| \\
&\leq& \max_{u^{\T}=(u_1^{\T},\cdots,u_n^{\T}):\sum_{i=1}^n\|u_i\|^2=1}\left|\sum_{i=1}^n\sum_{j=1}^n(A_{ij}-\mathbb{E}A_{ij})u_i^{\T}u_j\right| \\
&\leq& \opnorm{(A-\mathbb{E}A)\otimes I_d} \\
&\leq& \opnorm{A-\mathbb{E}A} \\
&\leq& C_1\sqrt{np},
\end{eqnarray*}
with probability at least $1-n^{-10}$ by Lemma \ref{lem:ER-graph}. The second term can be bounded by Lemma \ref{lem:bandeira}. That is,
$$\opnorm{(A\otimes \mathds{1}_d\mathds{1}_d^T)\circ W}\leq C_2\sqrt{npd},$$
with probability at least $1-n^{-10}$. Combining the above bounds, we have
$$\frac{1}{n^2}\fnorm{\wh{Z}\wh{Z}^{\T}-Z^*Z^{*\T}}^2\leq C_3\frac{d\left(1+\sigma^2d\right)}{np},$$
with probability at least $1-2n^{-10}$. Apply Lemma \ref{lem:loss-re}, and we have
\begin{equation}
\min_{B\in\o(d)}\frac{1}{n}\sum_{i=1}^n\fnorm{\wh{Z}_i-Z_i^*B}^2 \leq C_3\frac{d\left(1+\sigma^2d\right)}{np}. \label{eq:bd-Z-hat-ini}
\end{equation}

Next, let us consider the setting of $\o(d)$ synchronization. For any $B\in\o(d)$, we have
\begin{eqnarray}
\nonumber && \frac{1}{n}\sum_{i=1}^n\fnorm{Z_i^{(0)}-Z_i^*B}^2 \\
\nonumber &=& \frac{1}{n}\sum_{i=1}^n\fnorm{Z_i^{(0)}-Z_i^*B}^2\mathbb{I}\{\det(\wh{Z}_i)\neq 0\} + \frac{1}{n}\sum_{i=1}^n\fnorm{Z_i^{(0)}-Z_i^*B}^2\mathbb{I}\{\det(\wh{Z}_i)= 0\} \\
\nonumber &\leq& \frac{1}{n}\sum_{i=1}^n\fnorm{\P(\wh{Z}_i)-Z_i^*B}^2\mathbb{I}\{\det(\wh{Z}_i)\neq 0\}  \\
\label{eq:free1} && + \frac{1}{n}\sum_{i=1}^n\fnorm{Z_i^{(0)}-Z_i^*B}^2\mathbb{I}\{\fnorm{\wh{Z}_i-Z_i^*B}\geq e^{d^{-1}\log 2}-1\} \\
\label{eq:free2} &\leq& \frac{4}{n}\sum_{i=1}^n\fnorm{\wh{Z}_i-Z_i^*B}^2 + \frac{4d}{n}\sum_{i=1}^n\frac{\fnorm{\wh{Z}_i-Z_i^*B}^2}{\left(e^{d^{-1}\log 2}-1\right)^2} \\
\nonumber &\leq& C_4d^3\frac{1}{n}\sum_{i=1}^n\fnorm{\wh{Z}_i-Z_i^*B}^2,
\end{eqnarray}
where (\ref{eq:free1}) is because $\fnorm{\wh{Z}_i-Z_i^*B}< e^{d^{-1}\log 2}-1$ implies $\det(\wh{Z}_i)\neq 0$ by Lemma \ref{lem:ipsen}, and (\ref{eq:free2}) is by Lemma \ref{lem:li} and Markov inequality. Taking minimum on both sides and applying (\ref{eq:bd-Z-hat-ini}), we have
$$\min_{B\in\o(d)}\frac{1}{n}\sum_{i=1}^n\fnorm{Z_i^{(0)}-Z_i^*B}^2\leq C_5\frac{d^4\left(1+\sigma^2d\right)}{np},$$
with probability at least $1-2n^{-10}$.

Finally, we consider the setting of $\so(d)$ synchronization. We know that $Z_i^*\in\so(d)$ for all $i\in[n]$. For any $B\in\so(d)$, it is clear that we also have $Z_i^*B\in\so(d)$, which implies $\det(Z_i^*B)>0$. Then,
\begin{eqnarray}
\nonumber && \frac{1}{n}\sum_{i=1}^n\fnorm{Z_i^{(0)}-Z_i^*B}^2 \\
\nonumber &=& \frac{1}{n}\sum_{i=1}^n\fnorm{Z_i^{(0)}-Z_i^*B}^2\mathbb{I}\{\det(\wh{Z}_i)> 0\} + \frac{1}{n}\sum_{i=1}^n\fnorm{Z_i^{(0)}-Z_i^*B}^2\mathbb{I}\{\det(\wh{Z}_i)\leq 0\} \\
\label{eq:free3} &\leq& \frac{1}{n}\sum_{i=1}^n\fnorm{\P(\wh{Z}_i)-Z_i^*B}^2\mathbb{I}\{\det(\wh{Z}_i)> 0\}  \\
\label{eq:free4} && + \frac{1}{n}\sum_{i=1}^n\fnorm{Z_i^{(0)}-Z_i^*B}^2\mathbb{I}\{\fnorm{\wh{Z}_i-Z_i^*B}\geq e^{d^{-1}\log 2}-1\} \\
\nonumber &\leq& \frac{4}{n}\sum_{i=1}^n\fnorm{\wh{Z}_i-Z_i^*B}^2 + \frac{4d}{n}\sum_{i=1}^n\frac{\fnorm{\wh{Z}_i-Z_i^*B}^2}{\left(e^{d^{-1}\log 2}-1\right)^2} \\
\nonumber &\leq& C_4d^3\frac{1}{n}\sum_{i=1}^n\fnorm{\wh{Z}_i-Z_i^*B}^2.
\end{eqnarray}
In (\ref{eq:free3}) we have used the fact that $Z_i^{(0)}=\bar{\P}(\wh{Z}_i)=\P(\wh{Z}_i)$ when $\det(\wh{Z}_i)> 0$, and (\ref{eq:free4}) is because $\fnorm{\wh{Z}_i-Z_i^*B}< e^{d^{-1}\log 2}-1$ implies $\det(\wh{Z}_i)> 0$ by Lemma \ref{lem:ipsen} and we know the fact that $\det(Z_i^*B)>0$. Taking minimum on both sides of the inequality, we have
\begin{equation}
\min_{B\in\so(d)}\frac{1}{n}\sum_{i=1}^n\fnorm{Z_i^{(0)}-Z_i^*B}^2 \leq C_4d^3\min_{B\in\so(d)}\frac{1}{n}\sum_{i=1}^n\fnorm{\wh{Z}_i-Z_i^*B}^2. \label{eq:haha1}
\end{equation}
For any $B\in\o(d)\backslash\so(d)$, we can define $\wt{B}$ by changing the sign of the last column of $B$, and then we have $\wt{B}\in\so(d)$. We also define $\wt{Z}_i$ by changing the sign of the last column of $\wh{Z}_i$.
Note that $\det(Z_i^*\wt{B})>0$, and thus
\begin{eqnarray}
\nonumber && \frac{1}{n}\sum_{i=1}^n\fnorm{Z_i^{(0)}-Z_i^*\wt{B}}^2 \\
\nonumber &=&\frac{1}{n}\sum_{i=1}^n\fnorm{Z_i^{(0)}-Z_i^*\wt{B}}^2\mathbb{I}\{\det(\wt{Z}_i)> 0\} + \frac{1}{n}\sum_{i=1}^n\fnorm{Z_i^{(0)}-Z_i^*\wt{B}}^2\mathbb{I}\{\det(\wt{Z}_i)\leq 0\} \\
\label{eq:free5} &\leq& \frac{1}{n}\sum_{i=1}^n\fnorm{\P(\wt{Z}_i)-Z_i^*\wt{B}}^2\mathbb{I}\{\det(\wt{Z}_i)> 0\}  \\
\label{eq:free6} && + \frac{1}{n}\sum_{i=1}^n\fnorm{Z_i^{(0)}-Z_i^*\wt{B}}^2\mathbb{I}\{\fnorm{\wt{Z}_i-Z_i^*\wt{B}}\geq e^{d^{-1}\log 2}-1\} \\
\nonumber &\leq& \frac{4}{n}\sum_{i=1}^n\fnorm{\wt{Z}_i-Z_i^*\wt{B}}^2 + \frac{4d}{n}\sum_{i=1}^n\frac{\fnorm{\wt{Z}_i-Z_i^*\wt{B}}^2}{\left(e^{d^{-1}\log 2}-1\right)^2} \\
\nonumber &\leq& C_4d^3\frac{1}{n}\sum_{i=1}^n\fnorm{\wt{Z}_i-Z_i^*\wt{B}}^2 \\
\label{eq:free7} &=& C_4d^3\frac{1}{n}\sum_{i=1}^n\fnorm{\wh{Z}_i-Z_i^*{B}}^2.
\end{eqnarray}
To see (\ref{eq:free5}), we first apply Lemma \ref{lem:P-bar-prop} to obtain $Z_i^{(0)}=\bar{\P}(\wh{Z}_i)=\bar{\P}(\wt{Z}_i)$, and then we have $\bar{\P}(\wt{Z}_i)={\P}(\wt{Z}_i)$ when $\det(\wt{Z}_i)>0$.
The bound (\ref{eq:free6}) is obtained because $\fnorm{\wt{Z}_i-Z_i^*\wt{B}}< e^{d^{-1}\log 2}-1$ implies $\det(\wt{Z}_i)> 0$ by Lemma \ref{lem:ipsen} and we also know the fact that $\det(Z_i^*\wt{B})>0$. The last equality (\ref{eq:free7}) is a direct consequence of the definitions of $\wt{Z}_i$ and $\wt{B}$. Taking minimum on both sides of the inequality, we have
\begin{equation}
\min_{B\in\so(d)}\frac{1}{n}\sum_{i=1}^n\fnorm{Z_i^{(0)}-Z_i^*B}^2 \leq C_4d^3\min_{B\in\o(d)\backslash\so(d)}\frac{1}{n}\sum_{i=1}^n\fnorm{\wh{Z}_i-Z_i^*B}^2. \label{eq:haha2}
\end{equation}
Combining (\ref{eq:haha1}), (\ref{eq:haha2}) and (\ref{eq:bd-Z-hat-ini}), we have
\begin{eqnarray*}
&& \min_{B\in\so(d)}\frac{1}{n}\sum_{i=1}^n\fnorm{Z_i^{(0)}-Z_i^*B}^2 \\
&\leq& \min\left(C_4d^3\min_{B\in\so(d)}\frac{1}{n}\sum_{i=1}^n\fnorm{\wh{Z}_i-Z_i^*B}^2,C_4d^3\min_{B\in\o(d)\backslash\so(d)}\frac{1}{n}\sum_{i=1}^n\fnorm{\wh{Z}_i-Z_i^*B}^2\right) \\
&=& C_4d^3\min_{B\in\o(d)}\frac{1}{n}\sum_{i=1}^n\fnorm{\wh{Z}_i-Z_i^*B}^2 \\
&\leq& C_4\frac{d^4\left(1+\sigma^2d\right)}{np},
\end{eqnarray*}
with probability at least $1-2n^{-10}$. The proof is complete.
\end{proof}

\subsection{Proofs of Lemma \ref{lem:Q-exist}, Lemma \ref{lem:trees} and Theorem \ref{thm:lower-sod}}\label{sec:pf-lower}

\begin{proof}[Proof of Lemma \ref{lem:Q-exist}]
We use a mathematical induction argument. First, since
$$1-\left(r_{12}^2+\cdots+r_{1d}^2\right)\geq 1-\frac{1}{64d^4},$$
it is clear that $s_{11}$ is well defined and satisfies $s_{11}\geq\frac{7}{8}$. Next, we study $s_{21}$ and $s_{22}$. The equation (\ref{eq:eq1}) can be written as
\begin{equation}
s_{21}=-\frac{r_{13}r_{23}+\cdots+r_{1d}r_{2d}}{s_{11}}-\frac{r_{12}}{s_{11}}s_{22}. \label{eq:s21}
\end{equation}
We plug (\ref{eq:s21}) into (\ref{eq:eq2}) and obtain a quadratic equation of $s_{22}$. A sufficient condition for a quadratic equation $ax^2+bx+c=0$ to have two real solutions is $ac<0$. For the quadratic equation of $s_{22}$, this condition is
\begin{equation}
1-(r_{23}^2+\cdots+r_{2d}^2)-\frac{(r_{13}r_{23}+\cdots+r_{1d}r_{2d})^2}{s_{11}^2}>0. \label{eq:ac1}
\end{equation}
Since $s_{11}\geq\frac{7}{8}$ and $\max_{1\leq a<b\leq d}|r_{ab}|\leq \frac{1}{8d^{5/2}}$, (\ref{eq:ac1}) clearly holds, and thus $s_{22}$ is well defined. By (\ref{eq:s21}), $s_{21}$ is also well defined. We know from (\ref{eq:eq2}) that $|s_{22}|\leq 1$, and therefore, by (\ref{eq:s21}), we have the bound
$$|s_{21}|\leq \frac{|r_{13}r_{23}+\cdots+r_{1d}r_{2d}|+|r_{12}|}{7/8}\leq \frac{1}{56d^4}+\frac{1}{7d^{5/2}}\leq \frac{1}{4d^2},$$
and thus
$$s_{22}^2=1-s_{21}^2-(r_{23}^2+\cdots+r_{2d}^2)\geq 1-\frac{1}{16d^3},$$
which implies $s_{22}\geq\frac{7}{8}$.

Suppose $\max_{1\leq b<a\leq k-1}|s_{ab}|\leq \frac{1}{4d^2}$ and $\min_{a\in[k-1]}s_{aa}\geq\frac{7}{8}$, now we study $s_{k1},s_{k2},\cdots,s_{kk}$. Define $\wt{Q}_{k-1}\in\mathbb{R}^{(k-1)\times (k-1)}$ to be a diagonal matrix with the same diagonal elements as $Q_{k-1}$. Then, we have
$$\opnorm{Q_{k-1}-\wt{Q}_{k-1}}\leq \|Q_{k-1}-\wt{Q}_{k-1}\|_{\ell_{\infty}}\leq \frac{1}{4d},$$
which implies
\begin{equation}
s_{\min}(Q_{k-1})\geq s_{\min}(\wt{Q}_{k-1})-\opnorm{Q_{k-1}-\wt{Q}_{k-1}}\geq \frac{7}{8}-\frac{1}{4d}\geq\frac{3}{4}. \label{eq:Q-stable}
\end{equation}
Thus, $Q_{k-1}$ is invertible and we can write (\ref{eq:eq1}) as
\begin{equation}
S_{k-1}=-s_{kk}Q_{k-1}^{-1}R_{k-1}-Q_{k-1}^{-1}v_{k-1}. \label{eq:sk-1}
\end{equation}
We plug (\ref{eq:sk-1}) into (\ref{eq:eq2}) and obtain a quadratic equation of $s_{kk}$. Similar to (\ref{eq:ac1}), a sufficient condition for the quadratic equation to have two real solutions is
\begin{equation}
1-\|Q_{k-1}^{-1}v_{k-1}\|^2-\left(r_{k,k+1}^2+\cdots+r_{kd}^2\right)>0. \label{eq:ac2}
\end{equation}
By (\ref{eq:Q-stable}) and $\max_{1\leq a<b\leq d}|r_{ab}|\leq \frac{1}{8d^{5/2}}$, we have $\|Q_{k-1}^{-1}v_{k-1}\|^2\leq (16/9)\|v_{k-1}\|^2\leq \frac{16/9}{64^2d^7}$ and $r_{k,k+1}^2+\cdots+r_{kd}^2\leq \frac{1}{64d^4}$, and therefore (\ref{eq:ac2}) holds. Thus, $s_{kk}$ is well defined. By (\ref{eq:sk-1}), $s_{k1},\cdots,s_{k,k-1}$ are also well defined. We know from (\ref{eq:eq2}) that $|s_{kk}|\leq 1$, and therefore, by (\ref{eq:sk-1}), we have the bound
$$\|S_{k-1}\|\leq \|Q_{k-1}^{-1}R_{k-1}\|+\|Q_{k-1}^{-1}v_{k-1}\|\leq \frac{4}{3}\left(\|R_{k-1}\|+\|v_{k-1}\|\right)\leq \frac{1}{6d^{2}}+\frac{1}{48d^{7/2}}\leq \frac{1}{4d^2}.$$
We also have
$$s_{kk}^2=1-\|S_{k-1}\|^2-\left(r_{k,k+1}^2+\cdots+r_{kd}^2\right)\geq 1-\frac{1}{16d^4}-\frac{1}{64d^4},$$
which implies $s_{kk}\geq \frac{7}{8}$. To summarize, we have shown that $s_{k1},s_{k2},\cdots,s_{kk}$ are well defined. Moreover, we have $\max_{1\leq b<a\leq k}|s_{ab}|\leq \frac{1}{4d^2}$ and $\min_{a\in[k]}s_{aa}\geq\frac{7}{8}$. Hence, we conclude that $Q(r)$ is well defined, $Q(r)\in\o(d)$ and $\max_{1\leq b<a\leq d}|s_{ab}|\leq \frac{1}{4d^2}$ and $\min_{a\in[d]}s_{aa}\geq\frac{7}{8}$.

To prove $Q(r)\in\so(d)$, it suffices to show $\det(Q(r))>0$. We define a diagonal matrix $\wt{Q}(r)\in\mathbb{R}^{d\times d}$ that has the same diagonal elements as $Q(r)$. We know that $\det(\wt{Q}(r))\geq \left(\frac{7}{8}\right)^d>0$. Since $\max_{1\leq b<a\leq d}|s_{ab}|\leq \frac{1}{4d^2}$ and $\max_{1\leq a<b\leq d}|r_{ab}|\leq \frac{1}{8d^{5/2}}$, we have $\opnorm{Q(r)-\wt{Q}(r)}\leq (4d)^{-1}$. By Lemma \ref{lem:ipsen} and , we have
$$\frac{|\det(Q(r))-\det(\wt{Q}(r))|}{|\det(\wt{Q}(r))|}\leq \left(\frac{8}{7}\opnorm{Q(r)-\wt{Q}(r)}+1\right)^d-1\leq \left(\frac{8}{7}\frac{1}{4d}+1\right)^d-1<1,$$
and therefore, we have $\det(Q(r))>0$, which implies $Q(r)\in\so(d)$.

Finally, we analyze the derivative of $Q(r)$ with respect to each $r_{ab}$. This is also done via a mathematical induction argument. First, by the formula of $s_{11}$, we have $\left|\frac{\partial s_{11}}{\partial r_{ab}}\right|\leq \frac{1}{7d^{5/2}}$. Suppose $\sqrt{\left\|\frac{\partial S_{l-1}}{\partial r_{ab}}\right\|^2+\left|\frac{\partial s_{ll}}{\partial r_{ab}}\right|^2}\leq 5$ for all $l\in[k-1]$, now we study $\frac{\partial S_{k-1}}{\partial r_{ab}}$ and $\frac{\partial s_{kk}}{\partial r_{ab}}$. We take derivatives of both sides of (\ref{eq:eq1}) and (\ref{eq:eq2}) with respect to $r_{ab}$, and obtain
\begin{equation}
Q_k\begin{pmatrix}
\frac{\partial S_{k-1}}{\partial r_{ab}} \\
\frac{\partial s_{kk}}{\partial r_{ab}}
\end{pmatrix}=\begin{pmatrix}
-\frac{\partial v_{k-1}}{\partial r_{ab}} -\frac{\partial Q_{k-1}}{\partial r_{ab}}S_{k-1} -s_{kk}\frac{\partial R_{k-1}}{\partial r_{ab}} \\
-\frac{1}{2}\frac{\partial(r_{k,k+1}^2+\cdots+r_{kd}^2)}{\partial r_{ab}}
\end{pmatrix}. \label{eq:deriv-basic}
\end{equation}
By the definitions of $v_{k-1}$ and $R_{k-1}$, we have $\left\|\frac{\partial v_{k-1}}{\partial r_{ab}}\right\|\leq \frac{1}{8d^{5/2}}$ and $\left\|s_{kk}\frac{\partial R_{k-1}}{\partial r_{ab}}\right\|\leq 1$. By the condition $\sqrt{\left\|\frac{\partial S_{l-1}}{\partial r_{ab}}\right\|^2+\left|\frac{\partial s_{ll}}{\partial r_{ab}}\right|^2}\leq 5$ for all $l\in[k-1]$, we have
$$\opnorm{\frac{\partial Q_{k-1}}{\partial r_{ab}}}\leq \left\|\frac{\partial Q_{k-1}}{\partial r_{ab}}\right\|_{\ell_{\infty}}\leq \sqrt{d}\max_{1\leq l\leq k-1}\sqrt{\left\|\frac{\partial S_{l-1}}{\partial r_{ab}}\right\|^2+\left|\frac{\partial s_{ll}}{\partial r_{ab}}\right|^2}+1\leq 5\sqrt{d}+1,$$
which implies
$$\left\|\frac{\partial Q_{k-1}}{\partial r_{ab}}S_{k-1}\right\|\leq \opnorm{\frac{\partial Q_{k-1}}{\partial r_{ab}}}\|S_{k-1}\| \leq (5\sqrt{d}+1)\sqrt{\frac{1}{4d}}.$$
We also have $\left|\frac{1}{2}\frac{\partial(r_{k,k+1}^2+\cdots+r_{kd}^2)}{\partial r_{ab}}\right|\leq \frac{1}{8d^{5/2}}$. By (\ref{eq:Q-stable}) and (\ref{eq:deriv-basic}), we have
$$\left\|\begin{pmatrix}
\frac{\partial S_{k-1}}{\partial r_{ab}} \\
\frac{\partial s_{kk}}{\partial r_{ab}}
\end{pmatrix}\right\|^2\leq \frac{16}{9}\left\|\begin{pmatrix}
-\frac{\partial v_{k-1}}{\partial r_{ab}} -\frac{\partial Q_{k-1}}{\partial r_{ab}}S_{k-1} -s_{kk}\frac{\partial R_{k-1}}{\partial r_{ab}} \\
-\frac{1}{2}\frac{\partial(r_{k,k+1}^2+\cdots+r_{kd}^2)}{\partial r_{ab}}
\end{pmatrix}\right\|^2\leq 25,$$
and therefore $\sqrt{\left\|\frac{\partial S_{l-1}}{\partial r_{ab}}\right\|^2+\left|\frac{\partial s_{ll}}{\partial r_{ab}}\right|^2}\leq 5$ also holds for $l=k$. The proof is complete.
\end{proof}

\begin{proof}[Proof of Lemma \ref{lem:trees}]
Without loss of generality, we consider the problem with $i=1$ and $j=2$. Let us first understand the likelihood function for the problem. Given the knowledge of $Z_3,\cdots,Z_n$, we can decompose the likelihood function as
\begin{eqnarray*}
p((A\otimes \mathds{1}_d\mathds{1}_d^T)\circ Y,A) &=& p(A)p(A_{12}Y_{12}|A)\left(\prod_{i=3}^np(A_{1i}Y_{1i}|A)p(A_{2i}Y_{2i}|A)\right) \\
&& \times \prod_{3\leq i<j\leq n}p(A_{ij}Y_{ij}|A).
\end{eqnarray*}
Note that the part of the above decomposition that depends on $Z_1$ or $Z_2$ is
\begin{equation}
p(A_{12}Y_{12}|A)\left(\prod_{i=3}^np(A_{1i}Y_{1i}|A)p(A_{2i}Y_{2i}|A)\right), \label{eq:lik}
\end{equation}
which is proportional to the product of the density functions of $\mn\left(Z_1Z_2^{\T},\sigma^2A_{12}I_d,I_d\right)$, $\mn\left(\left(\sum_{i=3}^nA_{1i}\right)Z_1,\sigma^2\left(\sum_{i=3}^nA_{1i}\right)I_d,I_d\right)$ and $\mn\left(\left(\sum_{i=3}^nA_{2i}\right)Z_2,\sigma^2\left(\sum_{i=3}^nA_{2i}\right)I_d,I_d\right)$.

With the parametrization $Z_1=Z_1(r)=Q(r)$ and $Z_2=Z_2(r)=Q(r')$, we can write $T=T(r,r')=Z_1(r)Z(r')^{\T}$, and the logarithm of (\ref{eq:lik}) as $\ell(r,r')$. For the simplicity of notation, we re-index $\{r_{ab}\}_{1\leq a<b\leq d}$ and $\{r'_{ab}\}_{1\leq a<b\leq d}$ by $\{r_l\}_{1\leq l\leq d(d-1)/2}$ and $\{r'_l\}_{1\leq l\leq d(d-1)/2}$, respectively. The order of the re-indexing process is not important. Then, the information matrix is given by
$$B=\mathbb{E}\begin{pmatrix}
\nabla_r\ell(r,r')\nabla_{r}\ell(r,r')^{\T} & \nabla_r\ell(r,r')\nabla_{r'}\ell(r,r')^{\T} \\
\nabla_{r'}\ell(r,r')\nabla_{r}\ell(r,r')^{\T} & \nabla_{r'}\ell(r,r')\nabla_{r'}\ell(r,r')^{\T}
\end{pmatrix},$$
where the expectation is induced by the distribution (\ref{eq:lik}). Define the Jacobians $G_1=\frac{\partial \vec(Z_1(r))}{\partial r}\in\mathbb{R}^{\frac{d(d-1)}{2}\times d^2}$ and $G_2=\frac{\partial \vec(Z_2(r'))}{\partial r'}\in\mathbb{R}^{\frac{d(d-1)}{2}\times d^2}$. By direct calculation, we can write $B=B_1+B_2\in\mathbb{R}^{d(d-1)\times d(d-1)}$, where
$$B_1=\frac{p}{\sigma^2}\begin{pmatrix}
G_1G_1^{\T} & G_1G_2^{\T} \\
G_2G_1^{\T} & G_2G_2^{\T}
\end{pmatrix},$$
and
$$B_2=\frac{(n-2)p}{\sigma^2}\begin{pmatrix}
G_1G_1^{\T} & 0 \\
0 & G_2G_2^{\T}
\end{pmatrix}.$$
Define
$$F=\begin{pmatrix}
(Z_2\otimes I_d)G_1^{\T} & (I_d\otimes Z_1)G_2^{\T}
\end{pmatrix}\in\mathbb{R}^{d^2\times d(d-1)}.$$
By Equation (11) of the paper \cite{gill1995applications}, we have
\begin{eqnarray}
\nonumber && \inf_{\wh{T}}\int\int \mathbb{E}_Z\fnorm{\wh{T}-Z_1(r)Z_2(r')^{\T}}^2\diff P(r)\diff P(r') \\
\label{eq:trees-l} &\geq& \int\int \Tr(J(r,r'))\diff P(r)\diff P(r')-I(P),
\end{eqnarray}
where $J(r,r')=FB^{-1}F^{\T}$ and $I(P)$ is the information of the distribution $P$ that will be elaborated later.

To analyze (\ref{eq:trees-l}), we first need to show $B$ is invertible so that $J(r,r')$ is well defined. For any unit vector $v\in\mathbb{R}^{\frac{d(d-1)}{2}}$, we have
$$
v^{\T}G_1G_1^{\T}v = \left\|\sum_{l=1}^{\frac{d(d-1)}{2}}v_l\frac{\partial \vec(Z_1(r))}{\partial r_l}\right\|^2 \geq \left\|\sum_{l=1}^{\frac{d(d-1)}{2}}v_l\frac{\partial r}{\partial r_l}\right\|^2 = \sum_{l=1}^{\frac{d(d-1)}{2}}v_l^2=1.
$$
The inequality above is because $r$ can be viewed as a sub-vector of $\vec(Z_1(r))$. Recall the definition of $\{s_{ab}\}_{1\leq b\leq a\leq d}$ in the parametrization of $Q(r)$, and we also have
\begin{eqnarray*}
v^{\T}G_1G_1^{\T}v &=& \left\|\sum_{l=1}^{\frac{d(d-1)}{2}}v_l\frac{\partial \vec(Z_1(r))}{\partial r_l}\right\|^2 \\
&=& \left\|\sum_{l=1}^{\frac{d(d-1)}{2}}v_l\frac{\partial r}{\partial r_l}\right\|^2 + \sum_{1\leq b\leq a\leq d}\left(\sum_{l=1}^{\frac{d(d-1)}{2}}v_l\frac{\partial s_{ab}}{\partial r_l}\right)^2 \\
&\leq& 1 + \sum_{1\leq b\leq a\leq d}\sum_{l=1}^{\frac{d(d-1)}{2}}\left|\frac{\partial s_{ab}}{\partial r_l}\right|^2 \\
&\leq& 1+\frac{25d^3}{2},
\end{eqnarray*}
where the last inequality is by Lemma \ref{lem:Q-exist}. Therefore, we have
$$1\leq s_{\min}(G_1G_1^{\T})\leq s_{\max}(G_1G_1^{\T})\leq 1+\frac{25d^3}{2},$$
and the same bounds also apply to $G_2G_2^{\T}$. We also have
$$\opnorm{G_1G_2^{\T}}\leq \left(\max_{\|v\|=1}\|G_1v\|\right)\left(\max_{\|v\|=1}\|G_2v\|\right)\leq 1+\frac{25d^3}{2}.$$
With the above bounds, we immediately have
\begin{eqnarray}
\label{eq:info1} \opnorm{B_2} &\leq& \frac{(n-2)p}{\sigma^2}\left(1+\frac{25d^3}{2}\right), \\
\label{eq:info2} s_{\min}(B_2) &\geq& \frac{(n-2)p}{\sigma^2}, \\
\label{eq:info3} \opnorm{B_1} &\leq& \frac{2p}{\sigma^2}\left(1+\frac{25d^3}{2}\right).
\end{eqnarray}
Therefore, under the condition that $d$ is bounded by a constant, we know that both $B_1+B_2$ and $B_2$ are invertible when $n$ is sufficiently large.

Now we study the first term of (\ref{eq:trees-l}).
We can lower bound $\Tr(J(r,r'))$ by
\begin{eqnarray*}
\Tr(J(r,r')) &\geq& \Tr(FB_2^{-1}F^{\T})-|\Tr(F((B_1+B_2)^{-1}-B_2^{-1})F^{\T})| \\
&\geq& \Tr(FB_2^{-1}F^{\T})\left(1-\fnorm{B_2^{1/2}(B_1+B_2)^{-1}B_2^{1/2}-I_2}\right).
\end{eqnarray*}
By the definitions of $B_2$ and $F$, we have
\begin{eqnarray*}
\Tr(FB_2^{-1}F^{\T}) &=& \frac{\sigma^2}{(n-2)p}\Tr\left((Z_2\otimes I_d)G_1^{\T}(G_1G_1^{\T})^{-1}G_1(Z_2^{\T}\otimes I_d)\right) \\
&& + \frac{\sigma^2}{(n-2)p}\Tr\left((Z_1\otimes I_d)G_2^{\T}(G_2G_2^{\T})^{-1}G_2(Z_1^{\T}\otimes I_d)\right) \\
&=& \frac{\sigma^2d(d-1)}{(n-2)p}.
\end{eqnarray*}
By (\ref{eq:info1}), (\ref{eq:info2}) and (\ref{eq:info3}), we have
$$\fnorm{B_2^{1/2}(B_1+B_2)^{-1}B_2^{1/2}-I_2}\leq \opnorm{B_2}\opnorm{B_2^{-1}}\opnorm{(B_1+B_2)^{-1}}\fnorm{B_1}\leq C_1\frac{d^7}{n}.$$
Hence, the first term of (\ref{eq:trees-l}) has the following lower bound,
\begin{equation}
\int\int \Tr(J(r,r'))\diff P(r)\diff P(r') \geq \left(1-\frac{C_2}{n}\right)\frac{\sigma^2d(d-1)}{(n-2)p}, \label{eq:trees-main}
\end{equation}
for some constant $C_2$ only depending on the bound of $d$.

Finally, we need to give an upper bound for $I(P)$. Let $\lambda(\cdot)$ be the density function of $P$, and then $I(P)$ is defined by
$$I(P)=\int \sum_{ikl}\frac{1}{\lambda(r)\lambda(r')}\left(\frac{\partial}{\partial r_k}K_{ik}(r,r')\lambda(r)\lambda(r')\right)\left(\frac{\partial}{\partial r_l}K_{il}(r,r')\lambda(r)\lambda(r')\right)\diff r\diff r',$$
where $K(r,r')=FB^{-1}$. Given the definition of $\lambda(\cdot)$, we have the bound
$$I(P)\leq C_3\left(\max_{r,r'}\max_{i,k}\left|\frac{\partial}{\partial r_k}K_{ik}(r,r')\right|+\max_{r,r'}\max_{i,k}|K_{ik}(r,r')|\right)^2,$$
where $C_3$ is some constant only depending on the bound of $d$ and the maximum is taken over all $r$ and $r'$ that satisfy $\max_{1\leq a<b\leq d}|r_{ab}|\leq \frac{1}{8d^{5/2}}$ and $\max_{1\leq a<b\leq d}|r'_{ab}|\leq \frac{1}{8d^{5/2}}$. Though this bound can be computed explicitly using formulas of matrix derivatives, we omit the long and tedious details here. Intuitively, each entry of $K(r,r')$ is a smooth function of $r$ and $r'$, and the orders of $\frac{\partial}{\partial r_k}K_{ik}(r,r')$ and $K_{ik}(r,r')$ only depend on that of $B_2$, since the contribution of $B_1$ is negligible. We thus have $I(P)\leq C_4\left(\frac{\sigma^2}{np}\right)^2$ for some constant $C_4$ only depending on the bound of $d$. By (\ref{eq:trees-l}) and (\ref{eq:trees-main}), we have
$$\inf_{\wh{T}}\int\int \mathbb{E}_Z\fnorm{\wh{T}-Z_1(r)Z_2(r')^{\T}}^2\diff P(r)\diff P(r')\geq \left(1-C_5\left(\frac{1}{n}+\frac{\sigma^2}{np}\right)\right)\frac{\sigma^2d(d-1)}{np},$$
and the proof is complete.
\end{proof}

\begin{proof}[Proof of Theorem \ref{thm:lower-sod}]
By Lemma \ref{lem:loss-re} and the same argument that leads to (\ref{eq:lower-ineq2}), we have
\begin{eqnarray*}
&& \inf_{\wh{Z}\in\so(d)^n}\sup_{Z\in\so(d)^n}\mathbb{E}_Z\bar{\ell}(\wh{Z},Z) \\
&\geq& \frac{1}{2n^2}\sum_{1\leq i\neq j\leq n}\int\left(\inf_{\wh{T}}\int\int \mathbb{E}_Z\fnorm{\wh{T}-Z_iZ_j^{\T}}^2\diff \Pi(Z_i)\diff \Pi(Z_j)\right)\prod_{k\in[n]\backslash\{i,j\}}\diff \Pi(Z_k).
\end{eqnarray*}
Since $\supp(\Pi)\subset\so(d)$ by Lemma \ref{lem:Q-exist}, the conclusion of Lemma \ref{lem:trees} also applies here, and thus we have
$$\inf_{\wh{T}}\int\int \mathbb{E}_Z\fnorm{\wh{T}-Z_i(r)Z_j(r')^{\T}}^2\diff P(r)\diff P(r')\geq \left(1-C\left(\frac{1}{n}+\frac{\sigma^2}{np}\right)\right)\frac{\sigma^2d(d-1)}{np}.$$
This leads to the desired result.
\end{proof}

\bibliographystyle{dcu}
\bibliography{reference}

\end{document}